\documentclass[11pt]{article}

\usepackage[left=30mm,right=30mm, top=30mm, bottom=30mm]{geometry}
\usepackage{graphicx,color}
\usepackage[utf8]{inputenc}
\usepackage{amsmath}
\usepackage{amssymb}
\usepackage{amsthm}
\usepackage{comment}
\usepackage{enumitem}
\setlist[enumerate]{parsep=0pt}
\usepackage{mathtools}
\usepackage{cases}
\usepackage{multicol}
\usepackage{tabularx}
\usepackage{authblk}
\usepackage{todonotes}
\usepackage[colorlinks]{hyperref}
\hypersetup{
	colorlinks=true,
	linkcolor=black,
	citecolor=black,
    urlcolor=black
}
\usepackage[backend=biber, style=ieee, doi=true, url=false, isbn=false, sorting=nyt, eprint=true]{biblatex}
\usepackage{tikz}
\usepackage{tikz-cd}
\usetikzlibrary{%
	matrix,%
	calc,%
	arrows%
}
\usetikzlibrary{arrows}

\renewcommand{\Re}{\operatorname{Re}}

\theoremstyle{plain}
\newtheorem{thm}{Theorem}[section]
\newtheorem{prop}[thm]{Proposition}
\newtheorem{lem}[thm]{Lemma}

\theoremstyle{definition}
\newtheorem{defi}[thm]{Definition}
\newtheorem{ex}[thm]{Example}
\newtheorem{assumption}[thm]{Assumption}

\theoremstyle{remark}
\newtheorem{remark}[thm]{Remark}

\newcommand{\C}{\mathbb C}
\newcommand{\R}{\mathbb R}
\newcommand{\N}{\mathbb N}

\newcommand{\K}{\mathbb K}

\newcommand{\dom}{\mathrm{dom}}
\newcommand{\V}{\mathcal{V}}

\newcommand{\dx}[1][x]{\,\mathrm{d}#1}
\newcommand{\eps}{\varepsilon}
\newcommand{\Real}{{\rm Re}}

\newcommand{\ran}{
\mathchoice%
    {\rm ran\,}
    {\rm ran\,}
    {\rm ran}
    {\rm ran}
}
\renewcommand{\ker}{
\mathchoice%
    {\rm ker\,}
    {\rm ker\,}
    {\rm ker}
    {\rm ker}
}

\newcommand{\dd}{\mathrm{d}}
\newcommand{\inv}{^{-1}}
\newcommand{\ainv}{^{-\ast}}

\renewcommand{\mid}{\;\vert\;}

\title{Differential-algebraic system nodes}

\author[1]{Mehmet Erbay}
\author[2]{Birgit Jacob}
\author[3]{Timo Reis}

\affil[1]{University of Wuppertal, Gaußstraße 20, 42119 Wuppertal, Germany, {\tt erbay@uni-wuppertal.de}}
\affil[2]{University of Wuppertal, Gaußstraße 20, 42119 Wuppertal, Germany, {\tt bjacob@uni-wuppertal.de}}
\affil[3]{Technische Universität Ilmenau, Weimarer Straße 25, 98693 Ilmenau, Germany, {\tt timo.reis@tu-ilmenau.de}}

\date{\today}

\begin{document}
\maketitle

\begin{abstract}
    Infinite-dimensional differential algebraic equations (short DAEs) with input and output are studied.
   The concepts of operator nodes and system nodes are extended to systems which additionally may include algebraic constraints.
    Extrapolation spaces are investigated for differential-algebraic equations, and solutions of the extrapolated DAE are characterized using augmented Wong sequences. The resulting theory is then applied to characterize infinite-dimensional port-Hamiltonian DAEs.
\end{abstract}

\section{Introduction}

Differential-algebraic systems of the form  
\begin{align}\label{eq:dae-fd}
    \begin{split}
        E\dot{x}(t) &= Ax(t) + Bu(t),\\
        y(t) &= Cx(t) + Du(t),
    \end{split}
\end{align}
with matrices $E, A\in\R^{n\times n}$, $B\in\R^{n\times m}$, $C\in\R^{p\times n}$, $D\in\R^{p\times m}$ form a fundamental class of systems in finite-dimensional control theory \cite{mehrmann_stykel_2006}.  
If $E\not=0$ is singular, the system \eqref{eq:dae-fd} incorporates both differential and algebraic relations and naturally arises in applications such as electrical networks, constrained mechanical systems, and coupled multiphysics models \cite{reis_stykel_pabtec_2010,mueller_linear_mechanical_descriptor_1997}.  
The finite-dimensional theory of DAEs is well established, covering aspects such as regularity, consistency, and input–output behaviour; see, for example, \cite{berger_reis_controllability_2013,Trenn2013,berger_reis_trenn_observability_2017}.

In the infinite-dimensional setting, the situation becomes considerably more intricate.  
Even in the case $E = I_X$ (the identity on $X$), the modeling of practically relevant systems, such as partial differential equations with boundary control and observation, requires the use of unbounded operators and an extrapolation space of the state space.  
A particularly elegant framework for treating such systems was introduced in \cite{staffans_well-posed_2005}, leading to the consideration of systems of the form
\begin{equation}\label{eq:odenode}
    \begin{pmatrix}
        x(t)\\ y(t)
    \end{pmatrix}
    =
    \begin{bmatrix}
        A\&B \\ C\&D
    \end{bmatrix}
    \begin{pmatrix}
        x(t) \\ u(t)
    \end{pmatrix},
\end{equation}
where the block operator
\begin{equation}
    M \coloneqq 
    \begin{bmatrix}
        A\&B \\ C\&D
\end{bmatrix}\label{eq:sysnode}
\end{equation}
possesses certain structural properties.  
Most importantly, $M$ is required to be closed, and the free dynamics must form a strongly continuous semigroup.  
An operator $M$ satisfying these properties is referred to as a~\emph{system node}.
This concept has proven highly effective in the study of various linear systems governed by partial differential equations with boundary or distributed control and observation \cite{PhReSc23,ReSc23a,ReSc25,FJRS23}.

The aim of this paper is to extend this framework to infinite-dimensional differential-algebraic systems in which the operator $E$ may be non-bijective and acts between different Hilbert spaces. 
More precisely, we consider systems of the form
\begin{equation}\label{eq:daenode}
    \begin{bmatrix}
        \tfrac{\dd}{\dd t}E & 0\\0 & I_Y
    \end{bmatrix}
    \begin{pmatrix}
         x(t)\\ y(t)
    \end{pmatrix}
    =
    \begin{bmatrix}
        A\&B \\ C\&D
    \end{bmatrix}
    \begin{pmatrix}
        x(t) \\ u(t)
    \end{pmatrix},
\end{equation}
where $E$ is a bounded linear operator mapping a Hilbert space $X$ into another Hilbert space $Z$, and an operator $\left[\begin{smallmatrix}
        A\&B \\ C\&D
    \end{smallmatrix}\right]$ with similar structure as in~\eqref{eq:sysnode}.



One essential difference between the standard class \eqref{eq:odenode} and their differential-algebraic counterpart~\eqref{eq:daenode}
is that, instead of considering strongly continuous semigroups to describe the free dynamics, one requires that the pair $(E,A)$
(the operator $A$ can be constructed from the differential-algebraic system node)
possesses a \emph{complex resolvent index}. 
That is, there exist $\omega \in \R$ and $C > 0$ such that 
$\{\lambda \in \C \mid \Re \lambda \geq \omega\} \subseteq \rho(E,A)$ and
\begin{equation*}
    \|(\lambda E - A)^{-1}\|_{L(Z,X)} \leq C\, |\lambda|^{p-1}, \qquad  \Re \lambda \geq \omega. 
\end{equation*}
Here, $\rho(E,A)$ denotes the set of all $\lambda \in \C$ for which $\lambda E - A$ is boundedly invertible.  
The concept of a complex resolvent index ensures the existence of solutions for initial values in a certain subspace and for sufficiently smooth right-hand sides. 
In particular, it guarantees the existence of an integrated semigroup, which can be used to characterize such solutions; see, for example, \cite{erbay_integrated_2025, melnikova_abstract_nodate, melnikova_properties_1996}.

This work is organized as follows. In Section \ref{section:extrapolation} we introduce the concept of extrapolation spaces for differential-algebraic equations of the form $\frac{\dd}{\dd t}Ex(t)=Ax(t)$ and extend $A$ to a bounded operator $A_{-1}$ on $X$, see \cite{engel_one-parameter_2000, TuWe09}, for the case $E=I_X$.
This lays the foundation for defining and proving various properties of a differential-algebraic operator node later on, such as splitting $A\&B \begin{psmallmatrix}
    x\\u
\end{psmallmatrix}$ into $A_{-1}x+Bu$ and defining the transfer function. 
In Section~\ref{section:solution} we  analyse solutions of the extrapolated DAE $\frac{\dd}{\dd t}Ex(t) = A_{-1}x(t)+Bu(t)$ and characterize the set of consistent initial values using augmented Wong sequences. Then we turn our attention to the main part of the work, namely introducing differential-algebraic operator nodes/system nodes in Section~\ref{section:algebraic-operator-nodes}. 
In Section~\ref{section:ph-dae-nodes} we deal with the extension of port-Hamiltonian system nodes to the DAE case and end our studies with an educational example.

\textbf{Notation.} 
For a complex Hilbert space $X$, the associated norm, inner product and identity operator are denoted by $\Vert \cdot \Vert_X$, $\langle \cdot, \cdot\rangle_X$ and $I_X$, respectively. If it is clear from the context, we simply write $\Vert \cdot\Vert$, $\langle \cdot, \cdot\rangle$ and $I$. 
If not mentioned otherwise we write $X^\ast$ and $\langle \cdot, \cdot\rangle_{X^\ast, X}$ to denote the (anti-)dual of $X$ and the canonical duality product. Again, if it is clear from the context, we may skip the subscript. Unless otherwise specified, Hilbert spaces are conventionally identified with their anti-duals.
For a $\omega\in \R$ we set $\C_{\Real >\omega}\coloneqq \{\lambda \in \C \, \vert \, \Real \lambda >\omega\}$.

For an interval $J\subseteq\R$ we write $C(J;X)$ and $C^{p}(J;X)$ (with $p\in\N$) for the space of continuous $X$-valued functions and the space of $p$-times continuously $X$-valued differentiable maps. 
Moreover, $L^{p}(J;X)$ ($1\le p\le\infty$) denotes the Lebesgue–Bochner space of $p$-integrable $X$-valued functions (all integrals are understood in the Bochner sense \cite{Diestel77}), and $H^{k}(J;X)\subseteq L^{2}(J;X)$ is the Sobolev space of $X$-valued functions whose weak derivatives up to order $k$ belong to $L^{2}(J;X)$ (for $k\in \N$). Moreover, we set
\begin{equation}\label{eq:H0}
    H^k_{0,l}([0,\infty); Z)\coloneqq \{ f \in H^k([0,\infty); Z) \, \vert \, f(0)=\ldots= f^{(k-1)}(0) = 0\}.
\end{equation}
$L(X,Z)$ denotes the \textit{space of bounded linear operators} mapping from $X$ to $Z$. Furthermore, we abbreviate $L(X)\coloneqq L(X,X)$. 
We denote the domain of a (not necessarily bounded) linear operator $A$ mapping from $X$ to $Z$ by $\dom(A)$. We call $A$ \textit{densely defined}, if $\dom(A)$ is dense in $X$, and \textit{closed}, if the graph of $A$, i.e.~$\{(x,Ax)\in X\times Z \,\vert \, x\in \dom(A)\}$, is a closed subspace of $X\times Z$.

First, for a closed and densely defined linear operator $A\colon \dom(A)\subseteq X\to Z$ we call $(\lambda I-A)\inv $ \textit{the resolvent of $A$} for all $\lambda \in \rho(A)\coloneqq \{\lambda \in \C \, \vert \, (\lambda I-A) \text{ is boundedly invertible}\}$. Now, let $E\in L(X,Z)$ and $A\colon \dom(A)\subseteq X\to Z$ closed and densely defined. 
We call $(\lambda E-A)\inv$ \textit{the (generalised) resolvent of $(E,A)$} for all $\lambda \in \rho(E,A)\coloneqq \{ \lambda \in \C \, \vert \, (\lambda E-A)\inv \text{ boundedly invertible}\}$.
For notational convenience, we define the \textit{right} and \textit{left resolvents} of the operator pair $(E,A)$ by
\begin{equation*}
    R_r(\lambda) \coloneqq (\lambda E - A)^{-1}E,
    \qquad
    R_l(\lambda) \coloneqq E(\lambda E - A)^{-1}.
\end{equation*}
For simplicity, we write $(E,A)$ to denote either the operator pair $E$ and $A$ or the corresponding DAE $\frac{\dd}{\dd t}Ex=Ax$, depending on the context.  
In the case $X=Z$, we call $E\in L(X)$ \textit{positive}, if $\langle x,Ex\rangle_X\geq 0$ for all $x\in X$ and we call $A\colon\dom(A)\subseteq X\to X$ \textit{dissipative}, if $\Real \langle x,Ax\rangle_X\leq 0$ for all $x\in \dom(A)$.

\section{Extrapolation spaces and solutions}\label{section:extrapolation}

\subsection[The Z-1 space]{The $Z_{-1}$ space} 

In this section, we generalize the concept of the extrapolation space for linear operators (see for example \cite{TuWe09, engel_one-parameter_2000, staffans_well-posed_2005}) to operator pairs $(E,A)$. We focus on operators that act between two different spaces, as these form the basis for DAEs and, subsequently, for differential algebraic operator nodes. We need the following assumption.

\begin{assumption}\label{assumption:1}\hfill
    \begin{enumerate}[label=({\alph*)}]
        \item $X$ and $Z$ are complex Hilbert spaces.
        \item $E$ is a bounded linear operator from $X$ to $Z$.
        \item $A\colon\dom(A)\subseteq X\to Z$ is closed and densely defined.
        \item $\rho(E,A)$ is not empty.
    \end{enumerate}
\end{assumption}

We start with the definition of the spaces $X_1$ and $Z_{-1}$, and some basic properties of these. The results presented here are generalisations of Propositions 2.10.1 to 2.10.3 from \cite{TuWe09}, where the case $X=Z$ and $E=I_X$ was investigated. 

Let $(E,A)$ fulfil Assumption \ref{assumption:1} and $\mu \in \rho(E,A)$. Then, we define the spaces $X_1$, and $Z_1^d$ by $\dom(A)$ and $\dom(A^\ast)$ with the norm 
\begin{align*}
    \Vert x \Vert_{X_1} &\coloneqq \Vert (\mu E-A) x \Vert_Z, \qquad x\in \dom(A),
    \intertext{and}
    \Vert z \Vert_{Z_1^d} &\coloneqq \Vert (\bar \mu E^\ast - A^\ast)z \Vert_{X}, \qquad z \in \dom(A^\ast),
\end{align*}
respectively. 
The boundedness of $E$, $E^\ast$, $(\mu E-A)\inv$ and $(\bar \mu E^\ast -A^\ast)\inv$ implies that $\Vert \cdot\Vert_{X_1}$ is equivalent to the graph norm of $A$, denoted by $\Vert \cdot\Vert_A$, and that $\Vert \cdot\Vert_{Z_1^d}$ is equivalent to the graph norm of $A^\ast$, denoted by $\Vert \cdot\Vert_{A^\ast}$.
Furthermore, we define $Z_{-1}$ as the completion of $Z$ with respect to the norm
\begin{equation*}
    \Vert z \Vert_{Z_{-1}} \coloneqq \Vert (\mu E-A)\inv z \Vert_X, \qquad z\in Z.
\end{equation*}

\begin{lem}\label{lem:Z_{-1}-space}\hfill\\
    Let $(E,A)$ satisfy Assumption \ref{assumption:1}.
    Then, all norms $\Vert \cdot\Vert_{Z_{-1}}$ defined for individual $\mu\in\rho(E,A)$ are equivalent. Hence, $\Vert \cdot\Vert_{Z_{-1}}$ and $Z_{-1}$ are independent of the choice of $\mu$. Furthermore, $Z_{-1}$ is dual to $Z_1^d$ with respect to the pivot space $Z$.
\end{lem}

\begin{proof}
    For $z\in Z$ we have
    \begin{align*}
        \Vert z \Vert_{Z_{-1}} = \sup_{\substack{x\in X\\ \Vert x\Vert_X\leq 1}} \vert \langle (\mu E-A)\inv z, x\rangle_{X} \vert 
        = \sup_{\substack{x\in X\\ \Vert x\Vert_X\leq 1}} \vert \langle z, (\bar \mu E^\ast -A^\ast)\inv x\rangle_{Z}\vert 
        = \sup_{\substack{y\in Z_1^d\\ \Vert y \Vert_{Z} \leq 1}} \vert \langle z, y\rangle_{Z}\vert.
    \end{align*}
    Hence, $\Vert \cdot\Vert_{Z_{-1}}$ is the dual norm of $\Vert \cdot\Vert_{Z_1^d}$ with respect to the pivot space $Z$. Since $\Vert \cdot\Vert_{Z_1^d}$ does not depend on the choice of $\mu$, the same holds for $\Vert \cdot\Vert_{Z_{-1}}$ and, thus, for $Z_{-1}$.
\end{proof}

Lemma \ref{lem:Z_{-1}-space} allows us to extend $A$ continuously to the whole space $X$.

\begin{lem}\label{lem:bounded-extension-of-A}\hfill\\
    Let $(E,A)$ satisfy Assumption \ref{assumption:1} and let $\mu \in \rho(E,A)$. Then $E\in L(X,Z_{-1})$ and $A$ has a unique extension $A_{-1}\in L(X,Z_{-1})$. Furthermore, $\mu \in \rho(E,A_{-1})$ and $(\mu E-A_{-1})\inv \in L(Z_{-1}, X)$.
\end{lem}

\begin{proof}
    Clearly, $E\in L(X,Z_{-1})$.
    Furthermore, $A\in L(X_1, Z)$ and $(\mu E-A)\colon X_1\to Z$ is boundedly invertible for all $\mu \in \rho(E,A)$. By Lemma \ref{lem:Z_{-1}-space}, $A^\ast \in L(Z_1^d, X)$ and $Z_{-1}$ is dual to $Z_1^d$ with pivot space $Z$. Let $A_{-1}\in L(X,Z_{-1})$ denote the dual of $A^\ast$. Then for $x\in \dom(A)$ and $z\in \dom(A^\ast)$ we have
    \begin{equation*}
        \langle A_{-1}x,z\rangle_{Z_{-1}, Z_1^d} = \langle x, A^\ast z \rangle_{X,X} = \langle Ax,z\rangle_{Z,Z}.
    \end{equation*}
    Thus, $A_{-1}x=Ax$ for all $x\in \dom(A)$ and, since $\dom(A)$ is dense in $X$, $A_{-1}$ is the unique extension of $A$. 
    Define $R(\mu)\coloneqq (\mu E-A)\inv \in L(Z,X)$. Since $\Vert R(\mu) z \Vert = \Vert z \Vert_{Z_{-1}}$ for all $z\in Z$, $R(\mu)$ has a unique bounded extension $\tilde R(\mu)\in L(Z_{-1},X)$. Moreover, we have
    \begin{align*}
        \tilde R(\mu) (\mu E-A_{-1})x &=x, \quad x\in \dom(A),
        \intertext{and}
        (\mu E-A_{-1})\tilde R(\mu) z &= z, \quad z\in Z.
    \end{align*}
    Since $\dom(A)$ is dense in $X$ and $Z$ is dense in $Z_{-1}$, it follows that $(\mu E-A_{-1})\inv = \tilde R(\mu)$.
\end{proof}

The following growth condition for the resolvents of $(E,A)$ and $(E,A_{-1})$ will be useful in the next section. This growth condition is typically expressed in the literature in terms of the complex resolvent index.

\begin{defi}[Complex resolvent index]\hfill\\
    Let $(E,A)$ fulfil Assumption \ref{assumption:1}. The \textit{complex resolvent index} of $(E,A)$ is the smallest number $p=p_{\rm res}^{(E,A)} \in \N_0$, such that there exists a $\omega \in\mathbb{R}$, $C>0$ with $\mathbb{C}_{\mathrm{Re}>\omega}\subseteq \rho(E,A)$ and 
    \begin{equation}\label{eq:complex-resolvent-index}
        \Vert (\lambda E-A)^{-1} \Vert_{L(Z,X)} \leq C\vert\lambda\vert^{p-1},\quad \lambda\in\mathbb{C}_{\mathrm{Re}>\omega}.
    \end{equation}
\end{defi}

The existence of the complex resolvent index extends naturally to $(E,A_{-1})$.

\begin{prop}\label{prop:extended-complex-resolvent-index}\hfill\\
    Let $(E,A)$ satisfy Assumption \ref{assumption:1}. If $(E,A)$ has a complex resolvent index $p_{\mathrm{res}}^{(E,A)}$, then $(E,A_{-1})$ has a complex resolvent index as well, which is bounded by $p_{\mathrm{res}}^{(E,A)}+1$. 
\end{prop}

\begin{proof}
    Since $(E,A)$ has a complex index, there exist $C>0$ and $\omega \in \R$ such that \eqref{eq:complex-resolvent-index} holds. Let $\mu \in \rho(E,A)$. By \cite[Remark~2.1.3]{sviridyuk_linear_2003}, we have
    \begin{equation}
        (\lambda E-A_{-1})\inv = (\mu E-A_{-1})\inv + (\lambda-\mu) (\lambda E-A_{-1})\inv E (\mu E-A_{-1})\inv
    \end{equation}
    for all $\lambda \in \C_{\Real >\omega}$. Let $z\in Z_{-1}$. Then,  using \eqref{eq:complex-resolvent-index} and $\Vert z\Vert_{Z_{-1}}=\Vert (\mu E-A_{-1})\inv z\Vert_X$ we have
    \begin{align*}
        \Vert (\lambda E-A_{-1})\inv z \Vert_X &\leq \Vert (\mu E-A_{-1})\inv z \Vert_X + \vert \lambda -\mu\vert   \Vert (\lambda E-A)\inv E\Vert_{L(X)}   \Vert (\mu E-A_{-1})\inv z\Vert_X\\
        &\leq \Vert z\Vert_{Z_{-1}} + C \vert \lambda\vert^{p_{\mathrm{res}}^{(E,A)} -1}   \vert \lambda-\mu\vert   \Vert z\Vert_{Z_{-1}}.
    \end{align*}
    Note that we used $(\lambda E-A_{-1})\inv E = (\lambda E-A)\inv E$ in the first estimate. This holds as $E$ maps into $Z$ and the resolvent of $(E,A)$ and $(E,A_{-1})$ coincide on $Z$. Now,  choosing $\omega\in \R$ accordingly, it is clear that the resolvent of $(E,A_{-1})$ grows at most polynomially with degree $p_{\mathrm{res}}^{(E,A)}$ on $\C_{\Real >\omega}$ as an operator from $Z_{-1}$ to $X$. 
\end{proof}

\subsection{Solution concepts}\label{section:solution}
We focus on the DAE
\begin{equation}\label{eq:dae-1}
    \begin{split}
        \frac{\dd }{\dd t}Ex(t) &= A_{-1}x(t)+f(t), \quad t\geq 0,\\
        Ex(0)&= Ex_0,
    \end{split}
\end{equation}
with $x_0 \in X$ and $f\colon [0,\infty)\to Z_{-1}$. 

\begin{defi}[Solutions]\hfill\\
    Let $(E,A)$ fulfil Assumption \ref{assumption:1}.
    \begin{enumerate}[label=({\roman*)}]
        \item We call $x\colon [0,\infty) \to X$ a \textit{classical solution} of \eqref{eq:dae-1}, if $x\in C([0,\infty); X)$, $Ex\in C([0,\infty); Z)\cap C^1([0,\infty); Z_{-1})$, $x(t)\in \dom(A_{-1})=X$, $t\geq 0$ and $x$ solves \eqref{eq:dae-1}.
        
        \item We call $x\colon [0,\infty) \to X$ a \textit{mild solution} of \eqref{eq:dae-1}, if $x\in L^1([0,\infty); X)$, $Ex\in C([0,\infty); Z)$, $f\in L^1([0,\infty); Z_{-1})$ and
        \begin{equation}\label{eq:mild-sol}
            Ex(t) - Ex(0) = A_{-1}\int_0^t x(\tau)\dx[\tau] + \int_0^t f(\tau)\dx[\tau], \quad t\geq 0.
        \end{equation}

        \item We call $x\colon [0,\infty)\to X$ a \textit{weak solution} of \eqref{eq:dae-1}, if $Ex \in C([0,\infty); Z)$ with $Ex(0)=Ex_0$, $(t\mapsto\langle f(t), z^\ast\rangle_{Z_{-1}, Z_1^d} ) \in L^1([0,\infty))$, $(t\mapsto \langle Ex(t), z^\ast\rangle_{Z_{-1}, Z_1^d}) \in W^{1,1}([0,\infty))$ for all $z^\ast\in Z_1^d$, with weak derivative fulfilling 
        \begin{equation*}
            \frac{\dd}{\dd t} \langle Ex(t), z^\ast\rangle_{Z_{-1}, Z_1^d} = \langle x(t), A^\ast z^\ast\rangle_{X, X^\ast} + \langle f(t), z^\ast \rangle_{Z_{-1}, Z_1^d},
        \end{equation*}
        for almost all $t\in [0,\infty)$, $z^\ast \in \dom(A^\ast)$.
    \end{enumerate}
\end{defi}

These different solution concepts have been compared in \cite{gernandt_pseudo-resolvent_2023} for the finite time interval. Since the proof works analogously to the infinite time interval, we will omit it at this point. 

\begin{lem}\label{lem:solutions}\cite[Thm.~2.1]{gernandt_pseudo-resolvent_2023}\hfill\\
    Let $(E,A)$ fulfil Assumption \ref{assumption:1} Further, let $x_0 \in X$ and $f\colon [0,\infty) \to Z_{-1}$.
    \begin{enumerate}[label=\alph*)]
        \item If $f\in L^1([0,\infty); Z_{-1})$ and $x\colon [0,\infty)\to X$ is a classical solution of \eqref{eq:dae-1}, then $x$ is a mild solution of \eqref{eq:dae-1}.

        \item If $t \mapsto \langle f(t), z^\ast\rangle_{Z_{-1}, Z_1^d} \in L^1([0,\infty))$ for all $z^\ast \in Z_1^d$ and $x\colon [0,\infty)\to X$ is a classical solution of \eqref{eq:dae-1}, then $x$ is a weak solution.

        \item Let $f\in L^1([0,\infty); Z_{-1})$ and $x\colon [0,\infty)\to X$. Then, $x$ is a mild solution of \eqref{eq:dae-1} if and only if $x$ is a weak solution of \eqref{eq:dae-1}.

    \end{enumerate}
\end{lem}

Proposition \ref{prop:extended-complex-resolvent-index} shows that $(E,A_{-1})$ has a complex resolvent index if $(E,A)$ has one. Thus, we obtain the following result for inhomogeneities that vanish at $t=0$.

\begin{prop}\cite[Thm.~2.4 \& 2.6]{thesis-reis}\label{thm:inhomog-solutions-complex-resolvent-index}\hfill\\
    Let $(E,A)$ satisfy Assumption \ref{assumption:1} and have a complex resolvent index $p=p_{\mathrm{res}}^{(E,A_{-1})}+1$.
    \begin{enumerate}[label=({\alph*)}]
        \item If $f\in H_{0,l}^{p}([0,\infty); Z_{-1})$, then there exists a unique mild solution of \eqref{eq:dae-1} with $x_0=0$.
            
        \item If $f\in H_{0,l}^{p+1}([0,\infty); Z_{-1})$, then there exists a unique classical solution of \eqref{eq:dae-1}  with $x_0=0$.
    \end{enumerate}
\end{prop}

Proposition \ref{prop:extended-complex-resolvent-index} implies that $(E,A_{-1})$ has a complex resolvent index, which is bounded by $p\coloneqq p_{\mathrm{res}}+1$. It was shown in \cite{trostorff_semigroups_2020} that all mild solutions map into the intersection of all $\overline{\ran ((\lambda E-A_{-1})\inv E)^k}$, $k\in \N$ and that, if the complex resolvent index exists, this intersection coincides with $\overline{\ran ((\lambda E-A_{-1})\inv E)^{p}}$ (as this sequence of sets stabilizes at the given index).
Now, one might assume that the same holds for the extrapolated DAE \eqref{eq:dae-1}. However, in this case, it is even possible to show that this property holds using the index of the original DAE $\frac{\dd}{\dd t}Ex=Ax$.

\begin{lem}\label{lem:extrapolated-solution}\hfill\\
    Let $(E,A)$ satisfy Assumption \ref{assumption:1} and have a complex resolvent index. Every mild solution of \eqref{eq:dae-1} maps into $\overline{\ran ((\lambda E-A)\inv E)^{p^{(E,A)}_{\mathrm{res}}}}$. 
\end{lem}

\begin{proof}
    According to Lemma \ref{lem:bounded-extension-of-A} $(\lambda E-A_{-1})\inv$ coincides with $(\lambda E-A)\inv$ on $Z$ and, since $\ran E\subseteq Z$, one always has
    \begin{equation*}
        (\lambda E-A_{-1})\inv E = (\lambda E-A)\inv E, \quad \lambda \in \rho(E,A_{-1}).
    \end{equation*}
    In particular, $\overline{\ran ((\lambda E-A_{-1})\inv E)^k} = \overline{\ran ((\lambda E-A)\inv E)^k}$ stagnates at the complex resolvent index of $(E,A)$. Thus, the assertion follows from \cite[Prop.~4.1\&Prop.~5.1]{trostorff_semigroups_2020} and Proposition \ref{prop:extended-complex-resolvent-index}.
\end{proof}

\begin{remark}
    \begin{enumerate}[label=\alph*)]
        \item Even though all the mild/classical solutions $x\colon [0,\infty) \to X$ of $(E,A_{-1})$ map into the same space as the ones from $\frac{\dd}{\dd t}Ex=Ax$, it should be noted that these are not automatically mild/classical solutions of $\frac{\dd}{\dd t}Ex=Ax$. This follows from the fact that $\int_0^t x(s)\dx[s]$ (or $x(t)$) does not need to map into $\dom(A)$ for all $t\geq 0$.
        \item In \cite[Lem.~8.1]{gernandt_pseudo-resolvent_2023} it was shown that there is a direct transformation of classical solutions of $\frac{\dd}{\dd t}Ex=Ax$ to classical solutions of 
        \begin{equation}\label{eq:dae-pseudo-resolvente}
            \begin{split}
                \frac{\dd }{\dd t} R_k(\lambda) w(t) &= \left(\lambda R_k(\lambda) - I\right) w(t), \quad t\geq 0,\\
                R_k(\lambda) w(0) &= w_0,
            \end{split}
        \end{equation}
        with $k\in \{l,r\}$ and $R_r(\lambda) = (\lambda E-A)\inv E$, $R_l(\lambda) = E(\lambda E-A)\inv$ (for a fixed $\lambda \in \rho(E,A)$) for a suitably adapted initial value $w_0$. It is possible to generalise this for mild solutions, as long as one considers $R_r(\lambda)$, but, in general, it is not possible to show this for $R_l(\lambda)$. 
        Here, one can only show that every mild solution of \eqref{eq:dae-pseudo-resolvente} is also a mild solution of $\frac{\dd}{\dd t}Ex=Ax$, and not the other way round. Specifically, for the other direction one would have to take a solution $x$ of $\frac{\dd}{\dd t}Ex=Ax$ and show that $w\coloneqq (\lambda E-A)x$ is a solution of \eqref{eq:dae-pseudo-resolvente}. But since mild solutions do not always map into $\dom(A)$, $w$ is not always well-defined.
        
        This is where the extrapolated DAE $(E,A_{-1})$ comes into play. To be more precise, let $x\colon [0,\infty)\to X$ be a mild solution of \eqref{eq:dae-1} and set $w(t) \coloneqq (\lambda E-A_{-1}) x(t)$, which is well-defined as $x(t)\in \overline{ \ran ((\lambda E-A)\inv E)^{p_{\mathrm{res}}^{(E,A)}}}\subseteq \dom(A_{-1})=X$ by Lemma \ref{lem:extrapolated-solution}. Furthermore, we define $\tilde R_l(\lambda)\coloneqq E(\lambda E-A_{-1})\inv$ to denote the extrapolated left resolvent. Then 
        \begin{align*}
            \tilde R_l(\lambda) w(t) - \tilde R_l(\lambda) w(0) &= Ex(t)- Ex(0) = A_{-1}\int_0^t x(s)\dx[s] \\
            &= A_{-1}(\lambda E-A_{-1})\inv \int_0^t (\lambda E-A_{-1}) x(s)\dx[s]\\
            &= (\lambda \tilde R_l(\lambda)-I) \int_0^t w(s)\dx[s], \quad \text{ for } t \geq 0.
        \end{align*}
        Thus, $w$ is a mild solution of $\frac{\dd}{\dd t} \tilde R_l(\lambda) w(t) = (\lambda \tilde R_l(\lambda)-I)w(t)$, $\tilde R_l(\lambda) w(0) = w_0$.
    \end{enumerate}
\end{remark}

\subsection{Solvability of inhomogeneous DAEs}\label{subsection:extrapolated-dae-with-inhom}
We now consider the DAE
\begin{equation}\label{eq:dae-with-inhom}
    \begin{split}
        \frac{\dd}{\dd t}Ex(t) &= A_{-1}x(t)+f(t),\quad t\geq 0,\\
        Ex(0) &= Ex_0,
    \end{split}
\end{equation}
with inhomogeneity $f\colon [0,\infty)\to Z_{-1}$ and initial value $x_0\in X$.
Proposition \ref{thm:inhomog-solutions-complex-resolvent-index} provides a condition on the inhomogeneity under which a mild/classical solution exists. 
However, this condition requires the inhomogeneity and its derivatives to vanish at the initial time (i.e.~$f(0)=f'(0)=...=0)$ in addition to being sufficiently smooth. 
This restriction is problematic, as it is often coupled with additional constraints on the initial value $x_0$ (e.g.~$Ex_0 = 0$).

In the following, we relax the assumptions of Proposition \ref{thm:inhomog-solutions-complex-resolvent-index} to require only the smoothness of the inhomogeneity and examine the resulting effect on the initial condition $x_0$. Hence, we define the set of initial values paired with the inhomogeneities, for which \eqref{eq:dae-with-inhom} has a solution
\begin{equation*}
    \mathcal{C}_{(E,A_{-1})} \coloneqq \left\{ \begin{pmatrix}
        x_0 \\ f
    \end{pmatrix} \in X\times L^1([0,\infty); Z_{-1})\, \vert \, \exists\, \text{classical solution} \, x\colon[0,\infty)\to X \text{ of } \eqref{eq:dae-with-inhom}\right\}.
\end{equation*}

Obviously, we have $\begin{psmallmatrix}
    x_0 \\ f
\end{psmallmatrix} \in \mathcal{C}_{(E,A_{-1})}$ if and only if there exists a classical solution $x\colon [0,\infty)\to X\times Z_{-1}$ of \eqref{eq:dae-with-inhom} and $\begin{psmallmatrix}
    x\\f
\end{psmallmatrix}$ solves
\begin{equation}\label{eq:augmented-system}
    \frac{\dd}{\dd t}\begin{bmatrix}
        E & 0
    \end{bmatrix} \begin{pmatrix}
        x(t)\\ f(t)
    \end{pmatrix} = \begin{bmatrix}
        A_{-1} & I_{Z_{-1}}
    \end{bmatrix} \begin{pmatrix}
        x(t) \\ f(t)
    \end{pmatrix}, \quad t\geq 0, \qquad \begin{pmatrix}
        Ex(0) \\ f(0) 
    \end{pmatrix} = \begin{pmatrix}
        Ex_0 \\ f(0)
    \end{pmatrix}.
\end{equation}

To identify necessary and sufficient conditions on the inhomogeneity, we introduce the \textit{augmented Wong sequences}, which are generated by constructing the Wong sequence for the operator pair $(\begin{bsmallmatrix}
    E & 0
\end{bsmallmatrix}, \begin{bsmallmatrix}
    A_{-1} & I
\end{bsmallmatrix})$.

\begin{defi}[Augmented Wong sequence]\label{def:augmentend-wong-seq}\hfill\\
    Let $(E,A)$ satisfy Assumption \ref{assumption:1}. The sequence
    \begin{align*}
        \V_0 &\coloneqq X\times Z_{-1},\\
        \V_{i+1} &\coloneqq 
        \left\{\begin{pmatrix}
            x\\ z
        \end{pmatrix} \in X\times Z_{-1} \, \vert \, \begin{bmatrix}
            A_{-1} & I_{Z_{-1}}
        \end{bmatrix} \begin{pmatrix}
            x\\ z
        \end{pmatrix} \in \begin{bmatrix}
            E & 0
        \end{bmatrix}\left(\V_{i}\right)\right\}, \quad i \in \N_0,
    \end{align*}
    is called the \textit{augmented Wong sequence of $(E, A, I_{Z_{-1}})$}.
\end{defi}
The difference from the Wong sequence derived from a DAE is that this approach considers the augmented DAE, which does not has to be regular, i.e., the resolvent set $\rho(\begin{bsmallmatrix}
    E & 0
\end{bsmallmatrix}, \begin{bsmallmatrix}
    A_{-1} & I_{Z_{-1}}
\end{bsmallmatrix})$ can be empty. Further details on augmented Wong sequences, especially in finite dimensions, can be found in \cite{Berger2022, Lewis1986, Berger2014, Frankowska1990, Trenn2013}.

\begin{lem}\label{lem:augmented-wong-sequence}\hfill\\
    Let $(E,A)$ satisfy Assumption \ref{assumption:1}. Then
    \begin{equation*}
        \V_{i+1}\subseteq \V_i.
    \end{equation*}
\end{lem}

\begin{proof}
    Obviously $\V_1 \subseteq \V_0$ holds. Since $\V_2$ is the preimage of $\begin{bsmallmatrix}
        E & 0
    \end{bsmallmatrix}$ under $\begin{bsmallmatrix}
        A_{-1} & I
    \end{bsmallmatrix}$ we have
    \begin{equation*}
        \begin{bmatrix}
            A_{-1} & I_{Z_{-1}}
        \end{bmatrix}(\V_2)\subseteq \begin{bmatrix}
            E & 0
        \end{bmatrix}(\V_1) \subseteq \begin{bmatrix}
            E & 0
        \end{bmatrix}(\V_0).
    \end{equation*}
    Looking at the preimage of this subsetrelation, we obtain $\V_2\subseteq\V_1$. Thus, the assertion follows through inductive repetition of the same arguments.
\end{proof}

\begin{thm}\label{thm:necessary-cond}\hfill\\
    Let $(E,A)$ satisfy Assumption \ref{assumption:1}, $\begin{psmallmatrix}
        x_0 \\ f
    \end{psmallmatrix} \in \mathcal C_{(E,A_{-1})}$ and $x\colon [0,\infty)\to X$, such that $\begin{psmallmatrix}
        x\\ f
    \end{psmallmatrix}\colon [0,\infty)\to X\times Z_{-1}$ is a classical solution of \eqref{eq:augmented-system}. Then 
    \begin{equation*}
        \begin{psmallmatrix}
            x(t) \\ f(t)
        \end{psmallmatrix} \in \bigcap_{i\in \N_0}\overline{\V_i}, \quad t\geq 0.
    \end{equation*}
\end{thm}

\begin{proof}
    Let $\begin{psmallmatrix}
        x_0 \\ f
    \end{psmallmatrix} \in \mathcal C_{(E,A_{-1})}$ and $\begin{psmallmatrix}
        x\\ f
    \end{psmallmatrix}$ be a classical solution of \eqref{eq:augmented-system}.
    Lemma \ref{lem:solutions} implies that $\begin{psmallmatrix}
        x\\f
    \end{psmallmatrix}$ is also a mild solution, satisfying
    \begin{equation*}
        \begin{bmatrix}
            E & 0
        \end{bmatrix} \begin{pmatrix}
            x(t+h) - x(t) \\ f(t+h) - f(t)
        \end{pmatrix} = \begin{bmatrix}
            A_{-1} & I_{Z_{-1}}
        \end{bmatrix} \int_t^{t+h} \begin{pmatrix}
            x(s) \\ f(s)
        \end{pmatrix}\dx[s], \quad t\geq 0, h \geq 0.
    \end{equation*}
    Since $\begin{psmallmatrix}
        x(t+h) - x(t) \\ f(t+h) - f(t)
    \end{psmallmatrix} \in X\times Z_{-1} = \V_0$,
    $\int_t^{t+h} \begin{psmallmatrix}
        x(s) \\ f(s)
    \end{psmallmatrix}\dx[s]$ is in $\V_1$. Taking the limit $h\to 0$ yields $\begin{psmallmatrix} 
        x(t) \\ f(t)
    \end{psmallmatrix} \in \overline{\V_1}$. Repeating this argument shows that $\begin{psmallmatrix} 
        x(t) \\ f(t)
    \end{psmallmatrix} \in \overline{\V_i}$ for all $i\in \N_0$.
\end{proof}

\begin{thm}\label{thm:suff-cond}\hfill\\
    Let $(E,A)$ satisfy Assumption \ref{assumption:1}, let $(E,A)$ have a complex resolvent index, $p\geq p_{\mathrm{res}}^{(E,A_{-1})}+1$ and $f\in H^p([0,\infty); Z_{-1})$.
    If there are $x_0, x_1, \ldots, x_p\in X$ with
    \begin{equation}\label{eq:suff-cond}
        Ex_{j+1} = A_{-1} x_{j} + f^{(j)}(0), \qquad 0\leq j \leq p-1,
    \end{equation}
    then $\begin{psmallmatrix}
        x_0\\ f
    \end{psmallmatrix}\in \mathcal{C}_{(E,A_{-1})}$.
\end{thm}

\begin{proof}
    Let $s\in \rho(E,A) = \rho(E, A_{-1})$ with $\Real s >\omega$. 
    Adding $-sEx_{j}$ to both sides of \eqref{eq:suff-cond} and rearranging the terms, we obtain 
    \begin{equation*}
        E x_{j+1} + (sE-A_{-1}) x_{j} - f^{(j)}(0) = sE x_{j}, 
    \end{equation*}
    or equivalently,
    \begin{equation*}
        \frac{1}{s} (s E-A_{-1})\inv E x_{j+1} + \frac{1}{s} x_{j} - (s E-A_{-1})\inv \frac{1}{s} f^{(j)}(0) = (sE-A_{-1})\inv E x_{j}, \quad 0\leq j \leq p-1.
    \end{equation*}
    Successively substituting these equations into one another, we obtain
    \begin{equation}\label{eq:proof-augmented-wong-sequence}
        (sE-A_{-1})\inv Ex_{0} = \frac{1}{s^p} (sE-A_{-1})\inv E x_p + \sum_{j=0}^{p-1}\frac{1}{s^{j+1}} x_j - (s E-A_{-1})\inv \sum_{j=0}^{p-1} \frac{1}{s^{j+1}} f^{(j)}(0).
    \end{equation}
    Define
    \begin{equation*}
        \widehat x(s) = (sE-A_{-1})\inv \left( Ex_0 + \widehat f(s) \right), \quad s\in \C_{\Real s\geq \omega},
    \end{equation*}
    where $\widehat f$ denotes the Laplace-transform of $f$.
    We will show that $\widehat x$ is the Laplace transform of a classical solution $x$ of \eqref{eq:augmented-system}. 
    To do this, we separate $\widehat{x}$ as follows
    \begin{align*}
        \widehat x(s) &= \frac{1}{s^p} (sE-A_{-1})\inv E x_p + \sum_{j=0}^{p-1}\frac{1}{s^{j+1}} x_j - (s E-A_{-1})\inv \sum_{j=0}^{p-1} \frac{1}{s^{j+1}} Bu_j + (sE-A_{-1})\inv \widehat f(s)\\
        &= \underbrace{\frac{1}{s^p} (sE-A_{-1})\inv E x_p + \sum_{j=0}^{p-1}\frac{1}{s^{j+1}} x_j}_{\eqqcolon\widehat{x}_1(s)} + \underbrace{(s E-A_{-1})\inv  \left(\widehat f(s) -\sum_{j=0}^{p-1}\frac{1}{s^{j+1}} f^{(j)}(0) \right)}_{\eqqcolon \widehat{x}_2(s)}.
    \end{align*}
    Note that we used \eqref{eq:proof-augmented-wong-sequence} in the first equation.
    Let $\widehat{w}(s) \coloneqq \sum_{j=0}^{p-1}\frac{1}{s^{j+1}} f^{(j)}(0)$, which is the Laplace transform of $w(t) = \sum_{j=0}^{p-1} \frac{t^j}{j!}f^{(j)}(0)$.
    Then $\widehat{v}(s)\coloneqq \widehat{f}(s) - \widehat{w}(s)$, is the Laplace transform of $v(t)\coloneqq f(t)-w(t)$, which satisfies $v^{(j)}(0)=0$ for all $0\leq j \leq p-1$. 
    Hence, $v\in H^{p}_{0,l}([0,\infty);Z_{-1})$, and using Proposition \ref{thm:inhomog-solutions-complex-resolvent-index}, $x_2$ is a classsical solution of \eqref{eq:augmented-system} with $x_2(0)=0$, $v(0)=0$.
    
    To show that $x_1$ is a classical solution, we must show that $\widehat x_1(s)$ decays at least with rate $\frac{1}{s}$, which ensures that $x_1$ does not contain any distributions. 
    This holds because $(sE-A_{-1})\inv$ grows at most polynomially with rate $p-2$. Furthermore, computing $sE\widehat{x}_1(s) - A_{-1}\widehat x_1(s) = Ex_0 + \widehat{w}(s)$, it becomes clear that $x_1$ is a classical solution of 
    \begin{equation*}
        \begin{split}
            \frac{\dd}{\dd t}Ex(t) &= A_{-1}x(t) + w(t), \qquad t\geq 0,\\
            Ex(0) &= Ex_0, \quad w(0)=u_0.
        \end{split}
    \end{equation*}
    In summary, $x = x_1+x_2$ is a classical solution of \eqref{eq:augmented-system} and, hence, $\begin{psmallmatrix}
        x_0\\ f
    \end{psmallmatrix} \in \mathcal{C}_{(E,A_{-1})}$.
\end{proof}

\section{Differential-algebraic operator nodes}\label{section:algebraic-operator-nodes}
Let $X$, $Y$, $Z$ and $U$ be complex Hilbert spaces. We denote the canonical projection from $Z\times Y$ onto $Z$ and $Y$ by $P_Z$ and $P_Y$, respectively. Let 
\begin{equation*}
    S\colon \dom(S) \subseteq X\times U \to Z\times Y  
\end{equation*}
be a linear operator. We call $A\colon \dom(A) \subseteq X\to Z$ with $A x \coloneqq P_Z S \begin{psmallmatrix} x \\ 0 \end{psmallmatrix}$ and $\dom(A)\coloneqq \{ x\in X \, \vert \, \begin{psmallmatrix} x \\0 \end{psmallmatrix} \in \dom(S)\}$ the \textit{main operator} of $S$. We set
\begin{equation*}
    A\& B \coloneqq P_Z S\colon \dom(S) \subseteq X\times U \to Z\quad \mathrm{and} \quad C\& D \coloneqq P_Y S\colon \dom(S) \subseteq X\times U \to Y,
\end{equation*}
such that $S$ can be written as
\begin{equation*}
    S=\begin{bmatrix} A\& B \\ C\& D\end{bmatrix}.
\end{equation*}
For $X=Z$ this coincides with the notion of an operator node, which is used to examine properties of the following system $\begin{psmallmatrix} \frac{\dd}{\dd t} x(t) \\ y(t) \end{psmallmatrix} = S \begin{psmallmatrix} x(t) \\ u(t) \end{psmallmatrix}$, particularly its solutions \cite{PhReSc23, staffans_well-posed_2005}. 

Our aim is to extend this theory to include systems that do not solely act on a single space, i.e.~the main operator $A$ mapping from $X$ to $Z$. This generalization, among other benefits, enables us to describe and analyse DAEs, which frequently involve coupling multiple spaces and satisfying algebraic constraints across them.

Therefore, we define a bounded linear operator $E\colon X \to Z$ and associate the above defined $S$ with the system
\begin{equation}\label{eq:dae-system-node}
    \begin{bmatrix} \frac{\dd}{\dd t} E & 0\\ 0 & I_Y
    \end{bmatrix}\begin{pmatrix}
        x(t) \\ y(t) 
    \end{pmatrix} 
    = \begin{bmatrix}
        A\& B \\ C\& D
    \end{bmatrix}\begin{pmatrix} x(t) \\ u(t) \end{pmatrix}.
\end{equation}

\begin{defi}[$E$-operator node]\label{def:da-operator-node}\hfill\\
    An \textit{$E$-operator node} on $(X, U, Z, Y)$ is a linear operator $S\colon \dom(S)\subseteq X\times U \to Z\times Y$ together with a linear operator $E\colon X\to Z$ with the following properties
    \begin{enumerate}[label={\roman*)}]
        \item \label{def-i} $S$ is closed, $E$ is bounded,
        \item \label{def-ii} $A\& B$ is closed with $\dom(A\&B) = \dom(S)$,
        \item \label{def-iii} $\rho(E,A)$ is not empty and $\dom(A)$ is dense in $X$,
        \item \label{def-iv} for all $u \in U$ there exists an $x\in X$ such that $\begin{psmallmatrix} x \\u \end{psmallmatrix} \in \dom(A\& B)$.
    \end{enumerate}
    If, additionally, there are $p\in \N_0$, $M>0$ and $\omega \in \R$, such that $\C_{\Real \geq \omega}\subseteq \rho(E,A)$ and
    \begin{equation*}
        \left\Vert (\lambda E-A)\inv \right\Vert_{L(Z,X)} \leq M (1+\vert \lambda \vert^{p-1})
    \end{equation*}
    holds, i.e.~$(E,A)$ has a complex resolvent index, then the $E$-operator node is called an \textit{$E$-system node}.
\end{defi}

\begin{remark}
    Let $S$ be an $E$-operator node on $(X, U, Z, Y)$. Since $A\&B$ is closed, the main operator $A$ is closed as well. As $A$ is densely defined and $\rho(E,A)$ is not empty, we can define the spaces $X_1$ and $Z_{-1}$ as in Section \ref{section:extrapolation} and extend $A$ to an operator $A_{-1}\in L(X,Z_{-1})$. 
\end{remark}

$E$-operator nodes have the following properties,  see \cite[Ch.~4.7]{staffans_well-posed_2005} for the case $E=I_X$.

\begin{lem}\label{lem:properties-of-system-node}\hfill\\ 
    Let $S$ be an $E$-operator node. Then
    \begin{enumerate}[label={\alph*)}]
        \item \label{lem:properties-of-system-node-a}There exists a unique operator $B\in L(U,Z_{-1})$, such that $\begin{bmatrix} A_{-1} &  B\end{bmatrix}\colon X\times U \to Z_{-1}$ is an extension of $A\&B$ and 
        \begin{equation}\label{eq:domains}
            \dom(S) = \left\{ \begin{pmatrix} x \\u \end{pmatrix}\in X \times U \, \vert \, A_{-1}x+Bu \in Z \right\}.
        \end{equation}

        \item \label{lem:properties-of-system-node-b}The operator $C\&D$ is continuous from $\dom(S)$ with the graph norm to $Y$ and the operator $C\colon X_1\to Y$ defined by $Cx\coloneqq C\& D \begin{psmallmatrix}x\\0\end{psmallmatrix}$, $x\in X_1$, is continuous from $X_1$ to $Y$.
        
        \item \label{lem:properties-of-system-node-c}For all $u \in U$ there exists an $x\in X$ such that $\begin{psmallmatrix}x \\ u\end{psmallmatrix} \in \dom(A \& B)$ and $A\& B \begin{psmallmatrix}x \\ u\end{psmallmatrix} \in \ran E$. 

        \item \label{lem:properties-of-system-node-d} For every $\lambda \in \rho(E,A)$ the operator
        \begin{equation*}
            F_\lambda\coloneqq \begin{bmatrix}
                I & -(\lambda E-A_{-1})\inv B\\ 0 & I
            \end{bmatrix} \colon X\times U \to X\times U
        \end{equation*}
        is boundedly invertible and maps $\dom(S)$ one-to-one onto $X_1\times U$.
    \end{enumerate}
\end{lem}

\begin{proof}
    \begin{enumerate}[label={\alph*)}, listparindent=1.5em]
        \item Let $u\in U$. There exists an $x\in X$ such that $\begin{psmallmatrix}x \\u\end{psmallmatrix}\in \dom(S)$ by Definition \ref{def:da-operator-node} \ref{def-iv}. We define 
        \begin{equation*}
            Bu\coloneqq A\& B\begin{pmatrix}x \\u\end{pmatrix} -A_{-1}x.
        \end{equation*}
        Since $A\& B$ and $A_{-1}$ are both linear, $Bu$ is independent of the choice of $x$ and is therefore a well-defined operator $B\colon U\to Z_{-1}$. 
        
        Since $A\&B \begin{psmallmatrix}x \\u\end{psmallmatrix} = A_{-1}x+Bu \in Z$ for all $\begin{psmallmatrix}x \\u\end{psmallmatrix}\in \dom(S)$, we have $\dom(S)\subseteq \{ \begin{psmallmatrix} x \\u \end{psmallmatrix}\in X \times U \, \vert \, A_{-1}x+Bu \in Z \}$. Hence, $\begin{bsmallmatrix} A_{-1} &  B\end{bsmallmatrix}\colon X\times U \to Z$ is an extension of $A\&B$. 
        For the reverse inclusion, assume that $\begin{psmallmatrix} x \\u \end{psmallmatrix}\in X \times U$ with $A_{-1}x+Bu \in Z$. By Definition \ref{def:da-operator-node} \ref{def-iv}, there again exists an $\tilde x \in X$ such that $\begin{psmallmatrix} \tilde x \\u \end{psmallmatrix}\in \dom(S)$ and $A\&B\begin{psmallmatrix} \tilde x \\u \end{psmallmatrix} = A_{-1} \tilde x + Bu \in Z$. This implies $A_{-1}(x-\tilde x)\in Z$, $\begin{psmallmatrix} x-\tilde x \\0 \end{psmallmatrix} \in \dom(S)$ and, consequently, $\begin{psmallmatrix} x \\u \end{psmallmatrix}\in \dom(S)$.

        It remains to show that $B$ is bounded, which is done by proving that $B$ is closed and applying the closed graph theorem. 
        Let $(u_n)_n \subseteq U$ such that $u_n\to u$ and $Bu_n \to z$ in $Z_{-1}$. 
        Let $\lambda \in \rho(E,A)\subseteq \rho(E,A_{-1})$ and define $x_n \coloneqq (\lambda E-A_{-1})\inv Bu_n$, $n\in \N$. 
        Then $A_{-1} x_n + Bu_n = \lambda E(\lambda E-A_{-1})\inv Bu_n \in Z$, and therefore $\begin{psmallmatrix} x_n \\u_n \end{psmallmatrix} \in \dom(S)$. 
        Since $A\&B$ is closed, $\begin{psmallmatrix} x_n \\u_n \end{psmallmatrix} \to \begin{psmallmatrix} (\lambda E-A_{-1})\inv z \\u \end{psmallmatrix}$, and $A\&B \begin{psmallmatrix} x_n \\u_n \end{psmallmatrix} = A_{-1}x_n + Bu_n \to \lambda E (\lambda E-A_{-1})\inv z$ in $Z$, it follows that
        \begin{equation*}
            A\&B \begin{pmatrix} (\lambda E-A_{-1})\inv z \\u \end{pmatrix} = \lambda E (\lambda E-A_{-1})\inv z,
        \end{equation*}
        which implies $Bu = A\&B \begin{psmallmatrix} (\lambda E-A_{-1})\inv z \\u \end{psmallmatrix} - A_{-1}(\lambda E-A_{-1})\inv z = z$.

        \item \cite[Lem.~2.3]{PhReSc23} proves that $C\&D \in L(\dom(S), Y)$. As $C$ is the restriction of $C\& D$ to the closed subspace $X_1\times \{0\}$, it follows that $C\in L(X_1, Y)$.

        \item Let $u \in U$. Using Definition \ref{def:da-operator-node} \ref{def-iv}, there exists an $x \in X$ such that $\begin{psmallmatrix} x \\u \end{psmallmatrix}\in \dom(S)$. Set $\tilde x = (\lambda E-A_{-1})\inv Bu$. Then $A_{-1} \tilde x + Bu = \lambda E(\lambda E-A_{-1})\inv \tilde x\in \ran E \subseteq Z$. Thanks to Part \ref{lem:properties-of-system-node-a}, $\begin{psmallmatrix} \tilde x \\u \end{psmallmatrix} \in \dom(S)$ and $A\&B \begin{psmallmatrix} \tilde x \\u \end{psmallmatrix} =\lambda E(\lambda E-A_{-1})\inv \tilde x\in \ran E$.

        \item Clearly, $F_\lambda$ is invertible with inverse $F_\lambda\inv \coloneqq \begin{bsmallmatrix}I & (\lambda E-A_{-1})\inv B\\ 0 & I\end{bsmallmatrix}$. Both $F_\lambda$ and $F\inv_\lambda$ are bounded, as $(\lambda E-A_{-1})\inv B\in L(U,X)$. 
        Let $\begin{psmallmatrix}x\\u\end{psmallmatrix} \in \dom(S)$. Using Part \ref{lem:properties-of-system-node-a} we have $A_{-1}x+Bu \in Z$ and
        \begin{align*}
            x-(\lambda E-A_{-1})\inv Bu &= (\lambda E-A_{-1})\inv \left( (\lambda E -A_{-1})\inv x - Bu\right)\\
            &= (\lambda E-A_{-1})\inv \left( \lambda Ex - (A_{-1}x+Bu)\right) \in X_1.
        \end{align*}
        Thus, $F_\lambda \begin{psmallmatrix}x\\u\end{psmallmatrix}\in X_1\times U$. Conversely, assume that $\begin{psmallmatrix}x\\u\end{psmallmatrix}\in X_1 \times U$. Then
        \begin{align*}
            A_{-1}\left( x+(\lambda E-A_{-1})\inv Bu \right) + Bu &= Ax + A_{-1}(\lambda E-A_{-1})\inv Bu + Bu\\
            &= Ax + \lambda E(\lambda E-A_{-1})\inv Bu \in Z.
        \end{align*}
        Hence, $F_\lambda\inv \begin{psmallmatrix}x\\u\end{psmallmatrix} = \begin{psmallmatrix}x+(\lambda E-A_{-1})\inv Bu\\u\end{psmallmatrix} \in \dom(S)$. \qedhere
    \end{enumerate} 
\end{proof}

Based on Definition \ref{def:da-operator-node}, we shall define the following technical terms.

\begin{defi}[Control operator/observation operator/transfer function]\hfill\\
    Let $S$ be an $E$-operator node.
    \begin{enumerate}[label={\roman*)}, listparindent=1.5em]
        \item The operator $B$ defined in Lemma \ref{lem:properties-of-system-node} \ref{lem:properties-of-system-node-a} is called the \textit{control operator of $S$}.
        
        \item The operator $C$ defined in Lemma \ref{lem:properties-of-system-node} \ref{lem:properties-of-system-node-b} is called the \textit{observation operator of $S$}.
        
        \item The \textit{transfer function} of an $E$-operator node $S$ is the operator-valued function
        \begin{equation}\label{transfer-function}
            G(\lambda) \coloneqq C\& D \begin{bmatrix}(\lambda E-A_{-1})\inv B\\ I\end{bmatrix}, \quad \lambda \in \rho(E,A).
        \end{equation}
    \end{enumerate}
\end{defi}

\begin{remark}
    By Lemma \ref{lem:properties-of-system-node} \ref{lem:properties-of-system-node-d} the inverse of the operator $F_\lambda$ maps $\{0\}\times U$ on $\dom(S)$. Thus, for all $\lambda \in \rho(E,A)$ the transfer function $G(\lambda)\colon U\to Y$ is well-defined.
\end{remark}

\begin{lem}\label{lem:transfer-function}\hfill\\
    Let $S$ be an $E$-operator node. 
    \begin{enumerate}[label={\alph*)}, listparindent=1.5em]
        \item The transfer function of $S$ is analytic on $\rho(E,A)$ and for all $\lambda, \mu \in \rho(E,A)$
        \begin{align}\label{transfer-function-identity}
            \begin{split}
                G(\rho)-G(\lambda) &= C\left[ (\rho E-A_{-1})\inv - (\lambda E-A_{-1})\inv\right] B\\ 
                &= (\lambda - \rho) C (\rho E-A_{-1})\inv E (\lambda E-A_{-1})\inv B.
            \end{split}
        \end{align}

        \item For all $\lambda \in \rho(E,A)$ and all $\begin{psmallmatrix}x\\ u\end{psmallmatrix} \in \dom(S)$ 
        \begin{equation*}
            C\& D \begin{psmallmatrix}x \\ u\end{psmallmatrix} = C \left[ x - (\lambda E-A_{-1})\inv B u \right] + G(\lambda) u.
        \end{equation*}
    \end{enumerate}
\end{lem}

\begin{proof}
    \begin{enumerate}[label={\alph*)}, listparindent=1.5em]
        \item By \cite[Lem.~2.1.1]{sviridyuk_linear_2003}, the map $\lambda\mapsto (\lambda E-A_{-1})\inv$ is analytic, which implies that $G$ is analytic as well. The equation \eqref{transfer-function-identity} follows directly from the pseudo-resolvent identity \cite[Rem.~2.1.3]{sviridyuk_linear_2003}.

        \item By \eqref{transfer-function}, for all $\lambda \in \rho(E,A)$ and $\begin{psmallmatrix}x\\ u\end{psmallmatrix} \in \dom(S)$ we have
        \begin{equation*}
             C\& D \begin{pmatrix}x \\ u\end{pmatrix} - G(\lambda) u = C\& D \begin{pmatrix}x - (\lambda E-A_{-1})\inv B u \\u\end{pmatrix} = C \left[ x - (\lambda E-A_{-1})\inv B u \right]. \qedhere 
        \end{equation*}
    \end{enumerate}
\end{proof}

With Lemma \ref{lem:properties-of-system-node} and \ref{lem:transfer-function} it is possible to construct the dual of an $E$-operator node, as it has been done in \cite[Prop.~2.4]{malinen_when_2006} for operator nodes.

\begin{lem}\hfill\\
    Let $S$ be an $E$-operator node. Then, the adjoint of $S$ is given by
    \begin{equation*}
        S^\ast =\begin{bmatrix}
            [A\&B]^d\\ [C\&D]^d
        \end{bmatrix} \colon \dom(S^\ast)\subseteq Z\times Y \to X\times U,
    \end{equation*}
    with
    \begin{equation}
        \dom(S^\ast)\coloneqq \left\{ \begin{pmatrix}
            z\\y
        \end{pmatrix} \in Z\times Y \; \vert\;  A_{-1}^\ast z + C^\ast y \in X\right\},
    \end{equation}
    $[A\&B]^d \coloneqq [A_{-1}^\ast \; C^\ast]$ and for all $\lambda\in\rho(E^\ast, A^\ast)$
    \begin{equation*}
        [C\&D]^d \begin{pmatrix}
            z \\y
        \end{pmatrix} \coloneqq B^\ast \left(z-(\lambda E-A_{-1})\ainv C^\ast y\right) + G(\lambda)^\ast y, \quad \begin{pmatrix}
            z \\ y
        \end{pmatrix}\in \dom(S^\ast).
    \end{equation*}
    Moreover, $S^\ast$ is an $E^\ast$-operator node with main operator $A_{-1}^\ast$. 
\end{lem}

\begin{proof}
    Let $\lambda \in \rho(E^\ast, A_{-1}^\ast)$. By Lemma \ref{lem:properties-of-system-node} \ref{lem:properties-of-system-node-d}, $F_\lambda\inv \in L(X\times U)$ with $F_\lambda\inv(X_1\times U) = \dom(S)$. Thus, $SF_\lambda\inv$ is well-defined on $X_1\times U$ and $(SF_\lambda\inv)^\ast = F_\lambda\ainv S^\ast$. Let $\begin{psmallmatrix}
        x\\u
    \end{psmallmatrix} \in X_1\times U$. Then, by Lemma \ref{lem:properties-of-system-node},
    \begin{align*}
        SF_\lambda\inv \begin{pmatrix}
            x\\u
        \end{pmatrix} &= \begin{bmatrix}
            A\&B \\C\&D
        \end{bmatrix}\begin{bmatrix}
            I & (\lambda E-A_{-1})\inv B \\ 0 & I
        \end{bmatrix} \begin{pmatrix}
            x\\u
        \end{pmatrix}\\
        &= \begin{bmatrix}
            A_{-1}x + A_{-1}(\lambda E-A_{-1})\inv Bu + Bu\\ Cx + G(\lambda)u 
        \end{bmatrix} \\
        &= \begin{bmatrix}
            A_{-1}x + \lambda E(\lambda E-A_{-1})\inv Bu \\ Cx + G(\lambda)u
        \end{bmatrix} = \begin{bmatrix}
            A_{-1} & \lambda E (\lambda E-A_{-1})\inv B\\C & G(\lambda)
        \end{bmatrix}\begin{pmatrix}
            x \\ u
        \end{pmatrix}.
    \end{align*}
    Thus, 
    \begin{equation*}
        (SF_\lambda\inv)^\ast = \begin{bmatrix}
            A_{-1}^\ast & C^\ast\\ \bar\lambda B^\ast(\lambda E-A_{-1})\ainv E^\ast & G(\lambda)^\ast
        \end{bmatrix}
    \end{equation*}
    and $\dom(S^\ast)=\{\begin{psmallmatrix}
        z\\y
    \end{psmallmatrix} \in Z\times Y \; \vert \; A_{-1}^\ast z + C^\ast y \in X\}$. Hence, for all $\begin{psmallmatrix} z\\y \end{psmallmatrix} \in \dom(S^\ast)$,
    \begin{align*}
        S^\ast \begin{pmatrix}
            z\\y
        \end{pmatrix} &= F_\lambda^\ast (SF_\lambda\inv) \begin{pmatrix}
            z\\y
        \end{pmatrix} \\
        &= \begin{bmatrix}
            I & 0 \\ -B^\ast(\lambda E-A_{-1})\inv & I
        \end{bmatrix} \begin{pmatrix}
            A_{-1}^\ast z + C^\ast y\\ \bar\lambda B^\ast(\lambda E-A_{-1})\ainv E^\ast z + G(\lambda)^\ast y
        \end{pmatrix} \\
        &= \begin{pmatrix}
            A_{-1}^\ast z+C^\ast y\\ B^\ast (z-(\lambda E-A_{-1})\ainv C^\ast y) + G(\lambda)^\ast y
        \end{pmatrix} = \begin{bmatrix}
            [A\&B]^d\\ [C\&D]^d
        \end{bmatrix} \begin{pmatrix}
            z\\y
        \end{pmatrix}.
    \end{align*}
    It remains to show that $S^\ast$ is an $E^\ast$-operator node. As the first three properties follow directly, we only have to show that for all $y\in Y$, there exists a $z\in Z$ such that $\begin{psmallmatrix}
        z\\y
    \end{psmallmatrix} \in \dom([A\&B]^d) = \dom(S^\ast)$. This follows directly by choosing $z\coloneqq (\bar \lambda E^\ast-A_{-1}^\ast)\inv C^\ast y$ for an arbitrary $y\in Y$. This is in fact well-defined, as $C$ is a bounded operator from $X_1$ to $Y$ and, thus, $Y\subseteq \dom(C^\ast)$. Then
    \begin{align*}
        A_{-1}^\ast z+C^\ast y = A_{-1}^\ast  (\bar \lambda E^\ast-A_{-1}^\ast)\inv C^\ast y + C^\ast y = \bar \lambda E^\ast (\bar \lambda E^\ast-A_{-1}^\ast)\inv C^\ast \in X,
    \end{align*}
    as $\ran E^\ast\subseteq Z$. Thus, $\begin{psmallmatrix}
        z\\y
    \end{psmallmatrix}\in \dom([A\&B]^d)$. The statement that $S^\ast$ is an $E^\ast$-system node (under the assumption that $S$ is an $E$-system node) follows immediately, as the norm of a bounded operator coincides with the norm of its adjoint.
\end{proof}

In what follows, we define the classical trajectories of $E$-system nodes on an infinite time horizon. Note that this concept can be defined analogously for a finite time horizon $[0, T]$, for $T>0$, as it was done for system nodes in \cite{PhReSc23, staffans_well-posed_2005}.

\begin{defi}[Classical/generalized trajectories]\label{def:trajectory}\hfill\\ 
    Let $S$ be an $E$-system node.
    A \textit{classical trajectory} for \eqref{eq:dae-system-node} is 
    \begin{equation*}
        (x, u, Ex, y) \in  C([0,\infty); X) \times C([0,\infty); U) \times C^1([0,\infty); Z) \times C([0,\infty); Y),
    \end{equation*}
    which satisfies \eqref{eq:dae-system-node} for all $t\in [0,\infty)$.\\
    A \textit{generalized trajectory} for \eqref{eq:dae-system-node} is 
    \begin{equation*}
        (x, u, Ex, y) \in C([0,\infty); X) \times L^2([0,\infty); U) \times C([0,\infty); Z) \times L^2([0,\infty); Y),
    \end{equation*}
    which is a limit of classical trajectories for \eqref{eq:dae-system-node} on $[0,\infty)$ in the topology of $C([0,\infty); X) \times L^2([0,\infty); U) \times C([0,\infty); Z) \times L^2([0,\infty); Y)$.
\end{defi}

\begin{remark}\label{rem:trajectories}
    \begin{enumerate}[label={\alph*)}, listparindent=1.5em]
        \item As the third component $Ex$ in the trajectories is determined by $x$, we often write $(x,u,y)$ instead of $(x,u, Ex, y)$ to denote classical and generalized trajectories.
        
        \item\label{rem:trajectories-b} If $S=\begin{bsmallmatrix} A\& B \\ C\& D\end{bsmallmatrix}$ is an $E$-operator node on $(X, U, Z, Y)$, then $A\& B$ can be regarded as an $E$-operator node on $(X, U, Z, \{0\})$. Consequently, the system \eqref{eq:dae-system-node} reduces to
        \begin{equation}\label{eq:reduced-dae-system-node}
            \frac{\dd}{\dd t}Ex(t) = A\&  B \begin{pmatrix}
                x(t) \\ u(t)
            \end{pmatrix}.
        \end{equation}
    \end{enumerate}
\end{remark}

In the following, we give a condition for the existence of classical trajectories.

\begin{thm}\label{thm:existence-of-solutions-system-node}\hfill\\
    Let $S$ be an $E$-system node on $(X, U, Z, Y)$, $p=p_{\mathrm{res}}^{(E,A_{-1})}+2$, 
    $x_0\in X$, $u\in H^p([0,\infty);U)$ and $x_1, \ldots, x_p\in X$ with
    \begin{equation*}
        Ex_{j+1} = A_{-1} x_j + Bu^{(j)}(0), \qquad 0\leq j \leq p-1.
    \end{equation*}
    Then there exists a generalized trajectory $(x, u, y)$ for \eqref{eq:dae-system-node} with $x(0)=x_0$.
\end{thm}

\begin{proof}
    By Theorem \ref{thm:suff-cond}, $\begin{psmallmatrix}
        x_0 \\ f
    \end{psmallmatrix} \in \mathcal{C}_{(E,A_{-1})}$. Thus, there exists a classical solution $x\colon[0,\infty)\to X$ of 
    \begin{equation*}
        \begin{split}
            \frac{\dd}{\dd t}Ex(t) &= A_{-1}x(t) + Bu(t), \quad t \geq 0,\\
            Ex(0) &= Ex_0.
        \end{split}
    \end{equation*}
    Note that the solution $x$ constructed in the proof of Theorem \ref{thm:suff-cond} is at least two times differentiable, since $p=p_{\mathrm{res}}^{(E,A_{-1})}+2$. 
    As $E$ is bounded, $t\mapsto \frac{\dd}{\dd t} Ex = E\frac{\dd}{\dd t}x$ is continuous in $Z$. 
    Hence, $\frac{\dd}{\dd t}Ex(t) = A_{-1}x(t) + Bu(t)\in Z$ for all $t\geq 0$. 
    By Lemma \ref{lem:properties-of-system-node} \ref{lem:properties-of-system-node-a}, $\begin{psmallmatrix}
        x(t) \\ u(t)
    \end{psmallmatrix} \in \dom(S) = \dom(A\&B)$. Consequently, $x$ is also a classical solution of
    \begin{equation*}
        \begin{split}
            \frac{\dd}{\dd t}Ex(t) &= A\&B\begin{psmallmatrix}
                x(t) \\ u(t)
            \end{psmallmatrix}, \quad t \geq 0,\\
            Ex(0) &= E,x_0, \quad u(0)=u_0.
        \end{split}
    \end{equation*}
    Then, setting $y(t)= C\&D\begin{psmallmatrix}
        x(t) \\ u(t)
    \end{psmallmatrix}$, the triple $(x, u, y)$ is a classical trajectory of \eqref{eq:dae-system-node}.
\end{proof}

\section{Differential-algebraic port-Hamiltonian system nodes}\label{section:ph-dae-nodes}
In recent years, there has been considerable development in the formulation and study of solution to \textit{port-Hamil\-tonian differential-algebraic equations} (short \textit{pH-DAEs}) \cite{mehrmann_abstract_2023, mehl_spectral_2025}. 
The following finite-dimensional formulation serves as the basis for the subsequent modelling of pH-DAEs in this work
\begin{equation}\label{eq:intro-1}
    \begin{split}
        \frac{\dd}{\dd t} Ex(t) &= (J-R)Qx(t) + (B-P)u(t),\\
        y(t) &= (B^\ast + P^\ast)Qx(t) + (S-N)u(t),
    \end{split}
\end{equation}
see \cite{beattie_linear_2018, volker-riccardo-structure-preserving}, where $J\in \C^{n\times n}$ and $N\in \C^{m \times m}$ are skew-adjoint, $E^\ast Q \in \C^{n\times n}$ and $W\coloneqq \begin{bsmallmatrix}
    R & P\\ P^\ast & S
\end{bsmallmatrix}\in \C^{(n+m)\times (n+m)}$ are self-adjoint positive-semidefinite, for $E,Q, R \in \C^{n\times n}$, $B, P\in \C^{n\times m}$ and $S \in \C^{m\times m}$. In this case, the energy of \eqref{eq:intro-1}, also known as the \textit{Hamiltonian}, is given by $\mathcal{H}(x)=\frac{1}{2} \langle Ex, Q x\rangle_{\C^n}$. 
Furthermore, all solutions $x\colon [0,\infty)\to X$ of \eqref{eq:intro-1} satisfy the \textit{dissipation inequality}, that is, for all times $t>0$ and inputs $u\in L^2([0,t];\C^m)$ we obtain
\begin{equation*}
    \mathcal{H}(x(t))-\mathcal{H}(x(0)) \leq \Real \int_0^t \langle u(s), y(s)\rangle_{\C^m} \dx[s],
\end{equation*}
where $\mathcal{H}(x(t))$ represents the energy stored at time $t$ and $\Real \langle u(s), y(s)\rangle_{\C^m}$ denotes the external power supplied to the system.
This is shown using the skew-adjointness of $J$ and $N$, as well as the positive-semidefiniteness of $E^\ast Q$ and $W$.
Note that we will not go into further detail on \eqref{eq:intro-1}, as we will focus on infinite-dimensional systems in the course of this section.

The most general formulation of systems of the form \eqref{eq:intro-1} typically uses Dirac structures and Lagrangian submanifolds. 
These formulations involve vector spaces, which are self-orthogonal with respect to certain inner products. 
This enables the representation of pH-DAEs, as these geometric structures preserve properties such as the interconnection of multiple systems or the energy-dissipating properties, see \cite{arjan-dimitri-phs-introduction, mehrmann_differentialalgebraic_2023, schoberl_jet_2014, schaft_hamiltonian_2002, van_der_schaft_generalized_2018}.

The main disadvantage of such geometric approaches, at present, is that they do not address the existence and characterization of solutions.
This is where the previously introduced $E$-nodes become relevant. 

In what follows, we briefly recall the definitions of Gelfand triples and of dissipation nodes, which were intensively studied in \cite{skrepek_well-posedness_2021} and \cite{PhReSc23}. These concepts are used to introduce the notion of a \textit{port-Hamiltonian $E$-system node}. 
Therefore, let $X$, $Z$ and $U$ be complex valued Hilbert spaces and let $X^\ast$, $Z^\ast$ and $U^\ast$ be their anti-duals.

Let $h\colon \dom(h)\times \dom(h)\to \C$ be a densely defined, positive and closed sesquilinear form on $X$, i.e.~$\dom(h)$ is dense in $X$, $h(x,x)>0$ for all $x\in\dom(h)\backslash \{0\}$ and $\dom(h)$ is complete with respect to the norm 
\begin{equation*}
    \Vert x\Vert_{\dom(h)}\coloneqq \left( \Vert x \Vert_X^2 + h(x,x)\right)^{\frac{1}{2}}.
\end{equation*}
Then, there exists a unique positive self-adjoint operator $H\colon\dom(H)\subseteq X\to X$ such that $\dom(H^{\frac{1}{2}})=\dom(h)$ and
\begin{equation}\label{eq:energetic-extension}
    h(x,y) = \langle H^{\frac{1}{2}} x, H^{\frac{1}{2}} y\rangle, \quad x,y\in\dom(h).
\end{equation}

\begin{defi}[Quasi Gelfand triple]\hfill\\
    Let $X_h$ be the completion of $\dom(h)$ with respect to the norm
    \begin{equation*}
        \Vert x\Vert_h\coloneqq h(x,x)^{\frac{1}{2}}.
    \end{equation*}
    Furthermore, let $X_{h-}$ be the completion of $\{x' \in X \; \vert \; \Vert x'\Vert_{h-}<\infty\}$, where
    \begin{equation*}
        \Vert x'\Vert_{h-}\coloneqq \sup_{x\in \dom(h)\backslash \{0\}} \frac{\vert \langle x', x\rangle_X\vert}{\Vert x \Vert_h}.
    \end{equation*}
    Then $(X_{h-}, X, X_h)$ is called the \textit{quasi Gelfand triple associated with $h$}.
\end{defi}

In \cite[Cor.~4.8]{skrepek_well-posedness_2021} it is shown that $X_{h-}$ is isometrically isomorphic to $X_h^\ast$. Thus, we canonically identify $X_h^\ast=X_{h-}$ from now on and speak of quasi Gelfand triples of the form $(X_h^\ast, X, X_h)$.

\begin{lem}\cite[Prop.~2.17]{PhReSc23}\label{lem:unique-energetic-extension}\hfill\\
    The operator $H$ associated uniquely to a densely defined, closed and positive sesquilinear form $h$ has a unique bounded extension $\tilde H\colon X_h\to X_h^\ast$, which coincides with the Riesz isomorphism of $X_h$, denoted by $\mathcal{R}_{X_h}$.
\end{lem}

\begin{defi}[Dissipation node]\hfill\\
    Let $M=\begin{bsmallmatrix}
        F\&G \\K\& L
    \end{bsmallmatrix}$ be a $\mathcal{R}_{X_h}\inv$-operator node on $(X_h^\ast, U, X_h, U^\ast)$. Then $M$ is called \textit{dissipation node}, if $M$ is dissipative and $\rho(\mathcal{R}_{X_h}\inv,F)\cap \C_{\Real>0}$ is not empty.
\end{defi}

In this context, we call an operator $M\colon \dom(M)\subseteq X_h^\ast\times U\to X_h\times U^\ast$ \textit{dissipative}, if its graph is dissipative, i.e.~if 
\begin{equation*}
    \Real \left\langle \begin{pmatrix}
        x' \\ u
    \end{pmatrix},M \begin{pmatrix}
        x' \\ u
    \end{pmatrix}\right\rangle_{X_h^\ast\times U, X_h\times U^\ast}\leq 0, \quad \begin{pmatrix}
        x' \\ u
    \end{pmatrix}\in \dom(M).
\end{equation*}
Furthermore, $M$ is called \textit{maximal dissipative}, if it is dissipative and its graph is not a proper subset of a dissipative subspace of $(X_h^\ast\times U)\times(X_h\times U^\ast)$.

\begin{remark}
    Note, that we required in the definition of a dissipation node $\rho(\mathcal{R}_{X_h}\inv,F)\cap \C_{\Real>0}$, instead of $\lambda \mathcal{R}_{X_h}\inv-F$ having dense range for some $\lambda\in\C_{\Real>0}$, as is common in other literature. However, it can be shown that in this latter case, all elements in $\C_{\Real >0}$ are already contained in $\rho(\mathcal{R}_{X_h}\inv,F)$, see \cite[Prop.~3.5]{PhReSc23} and, thus, these two definitions are equivalent.
\end{remark}

In the following, we consider quasi Gelfand triples generated by sesquilinear forms $h_X$ on $X$ and $h_Z$ on $Z$. For notational simplicity, we leave out the index and denote the quasi Gelfand triples by $(X_{h}^\ast, X, X_h)$ and $(Z_{h}^\ast, Z, Z_h)$. 

\begin{defi}[Port-Hamiltonian $E$-system node]\hfill\\
    Let $(X_{h}^\ast, X, X_h)$, $(Z_{h}^\ast, Z, Z_h)$ be quasi Gelfand triples associated with $h_X$, $h_Z$.
    Further, $E,Q\in L(X_h, Z_h)$ and let 
    $M=\begin{bsmallmatrix}
        F\&G \\K\& L
    \end{bsmallmatrix}\colon \dom(M)\subseteq Z_h^\ast\times U\to Z_h\times U^\ast$ be a dissipation node on $(Z_h^\ast, U, Z_h, U^\ast)$. Then,
    \begin{equation}\label{eq:ph-operator-node}
        S=\begin{bmatrix}
            I_{Z_h} & 0\\ 0 & -I_{U^\ast} 
        \end{bmatrix} M \begin{bmatrix}
            R_{Z_h}Q & 0\\ 0 & I_U
        \end{bmatrix} = \begin{bmatrix}
            I_{Z_h} & 0\\ 0 & -I_{U^\ast} 
        \end{bmatrix} \begin{bmatrix}
            F\& G \\K\& L
        \end{bmatrix} \begin{bmatrix}
            \tilde H Q & 0\\ 0 & I_U
        \end{bmatrix} 
    \end{equation}
    is a \textit{port-Hamiltonian $E$-operator node} on $(X_h, U, Z_h, U^\ast)$, if 
    \begin{equation*}
        h_Z( Qx, Ex)\geq 0, \quad x\in X_h
    \end{equation*}
    and $\rho(E, F\tilde HQ)\cap \C_{\Real >0}$ is not empty. If $S$ is additionally an $E$-system node, then we call $S$ a \textit{port-Hamiltonian $E$-system node} on $(X_h, U, Z_h, U^\ast)$.
\end{defi}

The associated system for a port-Hamiltonian $E$-operator node reads as
\begin{equation}\label{eq:ph-dae-system-node}
    \begin{bmatrix}
        \frac{\dd}{\dd t} E & 0 \\ 0 & I_{U^\ast}
    \end{bmatrix} \begin{pmatrix}
        x(t) \\y(t)
    \end{pmatrix} = 
    \begin{bmatrix}
        F\& G \\ -(K\& L)
    \end{bmatrix}\begin{bmatrix}
        \tilde H Q & 0 \\ 0 & I_U
    \end{bmatrix} \begin{pmatrix}
        x(t) \\ u(t)
    \end{pmatrix}.
\end{equation}

\begin{defi}[Hamiltonian]\hfill\\
    Let $S=\begin{bsmallmatrix}
        I_{Z_h} & 0\\ 0 & -I_{U^\ast} 
    \end{bsmallmatrix} M \begin{bsmallmatrix}
        \tilde H Q & 0\\ 0 & I_U
    \end{bsmallmatrix}$ be a port-Hamiltonian $E$-operator node on $(X_h, U, Z_h, U^\ast)$. Then 
    \begin{align*}
        \mathcal{H}
        \colon X_h &\to \R,\\
        x &\mapsto \frac{1}{2} h_z(Qx, Ex) = \frac{1}{2} \langle \tilde H Qx, Ex\rangle_{Z_h^\ast, Z_h}
    \end{align*}
    is called the \textit{Hamiltonian of $S$}.
\end{defi}

\begin{prop}\label{prop:pH-dae-system-node}\hfill\\
    Let $S=\begin{bsmallmatrix}
        I_{Z_h} & 0\\ 0 & -I_{U^\ast} 
    \end{bsmallmatrix} M \begin{bsmallmatrix}
        \tilde H Q & 0\\ 0 & I_U
    \end{bsmallmatrix}$ be a port-Hamiltonian $E$-operator node on $(X_h, U, Z_h, U^\ast)$. If $Q$ is invertible, then
    \begin{equation*}
        \tilde S= \begin{bmatrix}
            I_{Z_h} & 0\\ 0 & -I_{U^\ast} 
        \end{bmatrix} \begin{bmatrix}
            F\& G \\K\& L
        \end{bmatrix} \begin{bmatrix}
            R_{Z_h} & 0 \\ 0 & I_U
        \end{bmatrix}
    \end{equation*}
    becomes a port-Hamiltonian $EQ\inv$-operator node on $(Z_h, U, Z_h, U^\ast)$. Further $S$ is a port-Hamiltonian $E$-system node, if and only if $\tilde S$ is a port-Hamiltonion $EQ\inv$-system node.
\end{prop}

\begin{proof}
    The first assumption follows from
    \begin{equation*}
        h_z(Qx, Ex) = h_z( z, EQ\inv z)\geq 0,
    \end{equation*}
    for arbitrary $x\in X_h$ with $Qx=z\in Z_h$, combined with the fact that $\rho(E,FQ)=\rho(EQ\inv, F)$. 
    Furthermore, the identity $(\lambda E-FQ)\inv = Q(\lambda E-F)\inv$ holds for all $\lambda \in \rho(E,FQ)$. 
    Thus, $(E,FQ)$ has a complex resolvent index (i.e.~the mapping $(\cdot E-FQ)\inv$ grows at most polynomially on a right half plane) if and only if $(EQ\inv, F)$ has a complex resolvent index. This implies the second assumption.          
\end{proof}

In Proposition \ref{prop:pH-dae-system-node} it was shown that one can always reduce a port-Hamiltonian $E$-operator node to a single quasi Gelfand triple, provided that $Q$ is invertible.

Next, we collect some properties of port-Hamiltonian $E$-operator nodes.

\begin{thm}\label{thm:index-and-dissipation-inequality}\hfill\\
    Let $S=\begin{bsmallmatrix}
        I_{Z_h} & 0\\ 0 & -I_{U^\ast} 
    \end{bsmallmatrix} M \begin{bsmallmatrix}
        R_{Z_h}Q & 0\\ 0 & I_U
    \end{bsmallmatrix}$ be a port-Hamiltonian $E$-operator node on $(X_h, U, Z_h, U^\ast)$.
    \begin{enumerate}[label={\alph*)}, listparindent=1.5em]
        \item If $Q$ is invertible and $\ran E$ is closed, then $\C_{\Real> 0}\subseteq \rho(E,FQ)$. In this case $S$ becomes a port-Hamiltonian $E$-system node and the complex resolvent index of $(E,FQ)$ is at most $3$. \label{thm:index-und-dissipation-inequality-a}

        \item All generalized trajectories $(x, u, y)$ (and thus also classical trajectories) satisfy the dissipation inequality
        \begin{align*}
            \mathcal{H}(x(t))-\mathcal{H}(x(0)) \leq \Real \int_0^t \langle y(\tau), u(\tau)\rangle_{U^\ast, U}\dx[\tau], \quad t\geq 0. 
        \end{align*}
    \end{enumerate}
\end{thm}

\begin{proof}
    \begin{enumerate}[label={\alph*)}, listparindent=1.5em]
        \item As seen in the proof of Proposition \ref{prop:pH-dae-system-node}, $EQ\inv$ is non-negative, as $h_Z(Qx, Ex) = h_Z(z, EQ\inv z)\geq 0$. Since the range of $E$ is closed, the same holds for $EQ\inv$ and, thus, we can decompose the space $Z_h$ as $Z_h=\ran EQ\inv \oplus \ker EQ\inv$. By \cite[Prop.~3.5]{PhReSc23}, $F$ is maximal dissipative. Hence, the pair $(EQ\inv, F)$ describes an abstract dissipative Hamiltonian differential-algebraic equation as defined in \cite[Assumption 9]{mehrmann_abstract_2023} and therefore, $\C_{\Real >0}\subseteq \rho(EQ\inv, F)$ by \cite[Thm.~13]{mehrmann_abstract_2023}. Using \cite[Prop.~5.1]{erbay_jacob_morris24} $(EQ\inv, F)$ has a complex resolvent index, which is at most $3$. Consequently, the assumption follows from Proposition \ref{prop:pH-dae-system-node}.

        \item Assume that $(x, u, y)$ is a classical trajectory in the sense of Definition \ref{def:trajectory}. Using the product rule and the dissipativity of $M=\begin{bsmallmatrix}
        F\&G \\K\& L
    \end{bsmallmatrix}$ we deduce for $t\geq 0$
        \begin{align*}
            \frac{\dd}{\dd t} \mathcal{H}(x(t)) &= \Real \langle \tilde H Qx(t), \frac{\dd}{\dd t}Ex(t)\rangle_{X_h^\ast, X_h} \\
            &= \Real \langle \tilde H Qx(t), F\& G\begin{pmatrix}
                \tilde H Q x(t)\\ u(t)
            \end{pmatrix}\rangle_{Z_h^\ast, Z_h} \\
            &= \Real \langle \begin{pmatrix}
                \tilde H Q x(t)\\ u(t)
            \end{pmatrix}, \begin{bmatrix}
                F\& G\\ K \& L
            \end{bmatrix} \begin{pmatrix}
                \tilde H Q x(t)\\ u(t)
            \end{pmatrix} \rangle_{Z_h^\ast\times U, Z_h\times U^\ast} \\
            & \quad + \Real \langle u(t), -K\&L  \begin{pmatrix}
                \tilde H Q x(t)\\ u(t)
            \end{pmatrix}\rangle_{U, U^\ast}\\
            &\leq \Real \langle u(t), y(t)\rangle_{U, U^\ast}.
        \end{align*}
            Integrating the above inequality yields the dissipation inequality. Analogously, the case of generalized solutions can be showed by taking limits. \qedhere
    \end{enumerate}
\end{proof}

\begin{remark}\label{rem:real-resolvent-index}
    Similar to Theorem \ref{thm:index-and-dissipation-inequality} \ref{thm:index-und-dissipation-inequality-a}, it can be shown that the \textit{real resolvent index} is bounded by $2$, see \cite[Prop.~5.1]{erbay_jacob_morris24}. The real resolvent index is defined similarly to the complex version, except that for some $\omega \in \R$ the interval $(\omega,\infty)$ must be contained in $\rho(E,A)$ and \eqref{eq:complex-resolvent-index} must hold for all $\lambda>\omega$.
\end{remark}

In most cases, DAEs arise from ordinary or partial differential equations that are modified by additional constrains, which typically take the form of algebraic conditions. Consequently, most port-Hamiltonian system nodes can be reformulated as port-Hamiltonian $E$-system nodes. In this context, for a Gelfand triple $X=Z$ with $h_X$, the operator 
\begin{equation*}
    S\coloneqq \begin{bmatrix}
            I_{X_h} & 0\\ 0 & -I_{U^\ast} 
        \end{bmatrix} M \begin{bmatrix}
            R_{X_h} & 0\\ 0 & I_U
        \end{bmatrix}
\end{equation*}
is called a \textit{port-Hamiltonian system node}, if $M=\begin{bsmallmatrix}
    F\& G\\ K\& L
\end{bsmallmatrix}\colon \dom(M)\subseteq X_h^\ast\times U \to X_h\times U^\ast$ is a dissipation node, i.e.~$\rho(R_{X_h}\inv, F)\cap \C_{\Real >0}$ is not empty, see \cite{PhReSc23}.

\begin{lem}\label{lem:ph-system-node}\hfill\\
    Let $X=Z$ be a quasi Gelfand triple associated with $h_X$, $E\in L(X_h, X_h)$ and let $S= \begin{bsmallmatrix}
            I_{X_h} & 0\\ 0 & -I_{U^\ast} 
        \end{bsmallmatrix} M \begin{bsmallmatrix}
            R_{X_h} & 0\\ 0 & I_U
        \end{bsmallmatrix}$ be a system node. If $h_X(x, Ex)\geq 0$ for all $x\in X_h$ and $\rho(E,F\tilde H)\cap \C_{\Real >0}$ is not empty, then $S$ is a port-Hamiltonian $E$-system node. 
\end{lem}

\begin{proof}
    This follows directly from the definition  of a port-Hamiltonian $E$-system node.
\end{proof}

We conclude this section by an example. 

\begin{ex}
    Let $\eps_1, \eps_2, r\in L^\infty(0,1)$ be non-negative and consider
    \begin{equation}\label{eq:example}
        \frac{\dd}{\dd t} \begin{bmatrix}
            \eps_1 (\zeta)& 0\\ 0 & \eps_2(\zeta)
        \end{bmatrix} 
        \begin{pmatrix}
            x_1(\zeta, t)\\ x_2(\zeta, t)
        \end{pmatrix} =
        \begin{bmatrix}
            0 & \frac{\partial}{\partial \zeta}\\ \frac{\partial}{\partial \zeta} & -r(\zeta)
        \end{bmatrix}
        \begin{pmatrix}
            x_1(\zeta, t)\\ x_2(\zeta, t)
        \end{pmatrix}, \quad t\geq 0, \zeta\in[0,1].
    \end{equation}
    We will show that this system fits into the framework of port-Hamiltonian $E$-system nodes. In addition, depending on the choice of additional conditions on $\eps_1$, $\eps_2$ and $r$, we will show that the properties of the system change significantly. To be more precise, by appropriately choosing these parameters, we can generate (almost) all possible indices of port-Hamiltonian $E$-system nodes, which, as we have already seen in Theorem \ref{thm:index-and-dissipation-inequality} \ref{thm:index-und-dissipation-inequality-a} and Remark \ref{rem:real-resolvent-index}, can be at most $3$ in the complex case and $2$ in the real case.

    Let $X=Z=L^2(0,1)\times L^2(0,1)$ and define
    \begin{equation*}
        E\coloneqq \begin{bmatrix}
            \eps_1 (\cdot)& 0\\ 0 & \eps_2(\cdot)
        \end{bmatrix}, \quad A\coloneqq \begin{bmatrix}
            0 & \frac{\partial}{\partial \zeta}\\ \frac{\partial}{\partial \zeta} & -r(\cdot)
        \end{bmatrix} 
    \end{equation*}
    with 
    \begin{equation*}
        \dom(A) \coloneqq \{ (x_1, x_2)\in H^1(0,1)\times H^1(0,1) \, \vert \, x_1(0)=x_2(1)=0\}.  
    \end{equation*}
    We choose the standard inner product in $L^2(0,1)\times L^2(0,1)$ as the sesquilinear form $h_X$. 
    Thus, the Hamiltonian is given through $\mathcal{H}(x) = \frac{1}{2}\langle x, Ex\rangle_{X_h} = \Vert Ex\Vert^2_{X_h}$ and $\dom(h_x) = X_h = X_h^\ast = L^2(0,1)\times L^2(0,1)$. 
    Further, the Riesz isomorphism is given by the identity operator. 

    For simplicity, we construct an $E$-operator node on $(X, \{0\}, X, \{0\})$, as seen in Remark \ref{rem:trajectories} \ref{rem:trajectories-b}, with main operator $A\& 0$ and set
    \begin{equation*}
        M\coloneqq \begin{bmatrix}
            A\& 0 \\0
        \end{bmatrix}.
    \end{equation*}
    We will show that this choice results in a port-Hamiltonian $E$-system node as in \eqref{eq:ph-operator-node}, which represents port-Hamiltonian DAEs without an input- or output-operator. It is, of course, possible to extend the system to include such operators by adjusting $M$, as long as the main operator is preserved and $M$ remains dissipative.
    Note that such systems may still have an input, which may be given by the algebraic constraints in $E$ or may appear in $\dom(A)$.

    We start by showing that $A$ is invertible on its domain and that $M$ is dissipative. Let $y=\begin{psmallmatrix}
        y_1\\ y_2
    \end{psmallmatrix}\in X_h$ and consider $A\begin{psmallmatrix}
        x_1 \\ x_2
    \end{psmallmatrix} = \begin{psmallmatrix}
        y_1\\ y_2
    \end{psmallmatrix}$ for $x=\begin{psmallmatrix}
        x_1 \\ x_2
    \end{psmallmatrix}$. Simple calculations uniquely determine $x$ as follows
    \begin{equation}\label{eq:example-inverse-von-A}
        \begin{split}
            x_1(\zeta) &= \int_0^\zeta y_2(s)\dx[s] + \int_0^\zeta r(s)\left(\int_s^1 y_1(v)\dx[v]\right)\dx[s],\\
            x_2(\zeta) &= -\int_\zeta^1 y_1(s)\dx[s].
        \end{split}
    \end{equation}
    Hence, $A$ is invertible and, by the closed graph theorem, its inverse is bounded. Furthermore, we have
    \begin{align*}
        \Real \langle x, Ax\rangle &= \Real \int_0^1 x_2'(s) \overline{x_1}(s) + x_1'(s) \overline{x_2}(s) - r(s) \vert x(s)\vert^2\dx[s]\\
        &= \Real \left[ x_2(s) \overline{x_1}(s)\right]_0^1 + \Real \int_0^1 x_1(s)\overline{x_2}(s) - x_2(s) \overline{x_1}(s)\dx[s] - \Real \int_0^1 r(s) \vert x_2(s)\vert^2 \dx[s],
    \end{align*}
    for all $x\in \dom(A)$. Due to the domain conditions $x_1(0)=x_2(1) = 0$, the first term (the boundary term) vanishes. The integrand of the second term $x_1\overline{x_2}- x_2\overline{x_1}$ is purely imaginary, so its real part is zero. The remaining term is non-positive, as $r$ is non-negative. Hence, $\Real\langle x, Ax\rangle\leq 0$, which means that $A$, and thus also $M$, is dissipative.
    
    Next, we show that $\C_{\Real > 0}$ is contained in $\rho(E,A)$. To do so, we only need to show that $\C_{\Real >0}\cap \rho(E,A)\neq \emptyset$, as the full inclusion then follows from \cite[Thm.~13]{mehrmann_abstract_2023}. Let $\lambda>0$ be arbitrary small. Since $A$ is invertible, we write
    \begin{equation*}
        (\lambda E-A) = -A(I-\lambda A\inv E).
    \end{equation*}
    Hence, $(\lambda E-A)$ is invertible if and only if $(I-\lambda A\inv E)$ is invertible. Choosing $\lambda >0$ small enough such that $\Vert \lambda A\inv E\Vert <1$, the Neumann series guarantees that $(I-\lambda A\inv E)$ is invertible with
    \begin{equation*}
        (I-\lambda A\inv E)\inv = \sum_{k=0}^\infty (\lambda A\inv E)^k.
    \end{equation*}
    Thus, $(\lambda E-A)$ is invertible for small $\lambda >0$, which proves that $\C_{\Real >0}\cap \rho(E,A)\neq \emptyset$.
    Lemma \ref{lem:ph-system-node} implies that
    \begin{equation*}
        S= \begin{bsmallmatrix}
            I_{X_h} & 0\\ 0 & -I_{U^\ast} 
        \end{bsmallmatrix} M \begin{bsmallmatrix}
            R_{X_h} & 0\\ 0 & I_U
        \end{bsmallmatrix}
    \end{equation*}
    is a port-Hamiltonian $E$-system node.
    Now, assuming that one or both functions $\eps_1$, $\eps_2$ either vanish or are strictly positive, we reduce the system to a DAE with index 0, 1, or 2.
    
    \textbf{Index 0.} Assume that $\eps_1$ and $\eps_2$ are strictly positive, i.e.~there exists a $c>0$ such that $\eps_1(\zeta)\geq c$, $\eps_2(\zeta)\geq c$ for all $\zeta\in [0,1]$. 
    In this case, \eqref{eq:example} becomes the wave equation. 
    Whether one wants to consider the damped wave equation or not, i.e.~whether $r$ is considered or not, does not affect the index of the system, as we will see later. 
    Let $x=\begin{psmallmatrix}
        x_1\\ x_2
    \end{psmallmatrix}\in\dom(A)$, $\lambda\in\C_{\Real >0}$. Then 
    \begin{align*}
        \Real \langle (\lambda E-A)x,x\rangle &= \Real \int_0^1 \lambda \eps_1(\zeta) \vert x_1(\zeta)\vert^2 + \left(\lambda \eps_2(\zeta) + r(\zeta) \right) \vert x_2(\zeta)\vert^2 \dx[\zeta] \\
        &\quad - \Real \int_0^1 x_2'(\zeta)\overline{x_1}(\zeta) + x_1'(\zeta)\overline{x_2}(\zeta)\dx[\zeta]\\
        &\geq \Real (\lambda)\, c \, \Vert x\Vert^2_{X_h}. 
    \end{align*}
    Note that the second integral vanished as $x_1(0)=x_2(1)=0$ and $\Real x_1\overline{x_2}- x_2\overline{x_1} = 0$. Substituting $(\lambda E-A)x = z$ into the resulting inequality, applying the Cauchy-Schwarz inequality on $\Real \langle z, (\lambda E-A)\inv z\rangle$ and dividing by $\Real (\lambda) c$ and s, we obtain
    \begin{equation*}
        \Vert (\lambda E-A)\inv z\Vert \leq \frac{1}{\Real(\lambda) c} \Vert z\Vert_{X_h}, \quad \lambda \in \C_{\Real >0}, z\in X_h.
    \end{equation*}
    Hence, it is clear that $(E,A)$ has a real resolvent index of $0$ and a complex resolvent index of $1$. 
    Note that the complex resolvent index cannot be $0$, otherwise the operator $A$ would be the generator of an analytic semigroup, which contradicts \eqref{eq:example} being the wave equation.

    \textbf{Index 1.} Assume that $\eps_1$ and $r$ are strictly positive, i.e.~there exists a $c>0$ such that $\eps_1(\zeta)\geq c$, $r(\zeta)\geq c$ for all $\zeta\in [0,1]$. 
    If $\eps_2=0$, then \eqref{eq:example} coincides with the diffusion equation.
    In this case the system \eqref{eq:example} can be rewritten as
    \begin{equation}\label{eq:diffusion-equation}
        \frac{\dd}{\dd t} x_1(\zeta, t) = \frac{\partial }{\partial \zeta}x_2(\zeta, t) = \frac{\partial}{\partial \zeta} \left(\frac{1}{r(\zeta)}\left( \frac{\partial}{\partial \zeta} x_1(\zeta, t)\right)\right), \quad t\geq 0, \zeta \in [0,1].
    \end{equation}
    To determine the index of $(E,A)$, we proceed analogously to the index $0$ case and obtain 
    \begin{equation*}
        \Real \langle (\lambda E-A)x ,x\rangle \geq c \; \Vert x\Vert^2_{X_h}
    \end{equation*}
    for $x=\begin{psmallmatrix}
        x_1 \\ x_2
    \end{psmallmatrix}\in \dom(A)$, $\lambda \in \C_{\Real >1}$. Hence,
    \begin{equation*}
        \Vert (\lambda E-A)\inv z\Vert \leq \frac{1}{c}\Vert z\Vert, \quad \lambda \in \C_{\Real >1}, z\in X_h.
    \end{equation*}
    Thus, the real and complex resolvent indices of $(E,A)$ are at most $1$. 
    Choosing $y\coloneqq \begin{psmallmatrix}
        0 \\ \frac{r}{\Vert r\Vert_{L^\infty}}
    \end{psmallmatrix}\in X$ and solving $(\lambda E-A)x=y$, we obtain $x = \begin{psmallmatrix}
        0 \\ \frac{1}{\Vert r \Vert_{L^\infty}}
    \end{psmallmatrix}$. 
    Hence, $x$ is independent of $\lambda$ and $\Vert (\lambda E-A)\inv y\Vert$ provides a $\lambda$-independent lower bound for $\Vert (\lambda E-A)\inv \Vert$.
    This means the resolvent is bounded from above and below independent of $\lambda$, which leads to both the real and complex resolvent indices being $1$.
    One might normally expect a lower index, since \eqref{eq:diffusion-equation} generates an analytic semigroup. However, this is not the case, as one must distinguish between the underlying spaces in \eqref{eq:example} and \eqref{eq:diffusion-equation}. 
    
    We obtain the same index if we set $\eps_1=\eps_2=0$. In this case, \eqref{eq:example} can be rewritten as the following elliptic partial differential equation
    \begin{equation*}
        0 = \frac{\partial}{\partial \zeta} \left(\frac{1}{r(\zeta)}\left( \frac{\partial}{\partial \zeta} x_1(\zeta, t)\right)\right), \quad t\geq 0, \zeta \in [0,1].
    \end{equation*}
    Then, $(\lambda E-A)\inv = A\inv$, which is given by \eqref{eq:example-inverse-von-A}. As the resolvent is indepentend of $\lambda$, the index can be seen directly.

    \textbf{Index 2.} Now, assume that $\eps_2$ is strictly positive and $\eps_1 = 0$. 
    Then \eqref{eq:example} reduces to
    \begin{align*}
        0 &= \frac{\partial}{\partial \zeta} x_2(\zeta,t),\\
        \frac{\dd}{\dd t} \eps_2(\zeta) x_2(\zeta, t) &= \frac{\partial}{\partial \zeta}x_1(\zeta, t)- r(\zeta)x_2(\zeta,t).
    \end{align*}
    To compute the index, we consider
    \begin{align*}
        (\lambda E-A) \begin{pmatrix}
            x_1(\zeta)\\ x_2 (\zeta)
        \end{pmatrix} = \begin{pmatrix}
            -\frac{\partial }{\partial \zeta} x_2(\zeta) \\ \lambda \eps_2(\zeta) x_2(\zeta) - \frac{\partial }{\partial \zeta}x_1( \zeta) + r(\zeta) x_2(\zeta)
        \end{pmatrix} = \begin{pmatrix}
            y_1(\zeta) \\ y_2(\zeta)
        \end{pmatrix},
    \end{align*}
    for $\begin{psmallmatrix}
        x_1\\ x_2
    \end{psmallmatrix} \in \dom(A)$, $\begin{psmallmatrix}
        y_1 \\ y_2
    \end{psmallmatrix} \in X$. This is solved by
    \begin{equation*}
        \begin{pmatrix}
            x_1(\zeta) \\ x_2(\zeta)
        \end{pmatrix} = \begin{pmatrix}
            \int_0^\zeta\left(\lambda \eps_2(s) + r(s)\right) \left(\int_s^1 y_1(k)\dx[k]\right)\dx[s]-\int_0^\zeta y_2(s)\dx[s]\\ \int_\zeta^1 y_1(s) \dx[s]
        \end{pmatrix}.
    \end{equation*}
    From this expression, it is clear that the solution $x$ depends linearly on $\lambda$. This implies that the resolvent growth is linear, and consequently, both the real and complex resolvent indices are $2$.
\end{ex}

\section*{Funding}
The authors gratefully acknowledge funding by the Deutsche Forschungsgemeinschaft (DFG, German Research Foundation) – Project-ID 531152215 – CRC 1701.

\section*{Declarations}

\subsection*{Data Availability Statement}
There are no data equipped with this article. Data sharing is therefore not applicable to this article.

\subsection*{Underlying and related material}
No underlying or related material.

\subsection*{Author contributions}
All authors contributed equally.

\section*{Competing interests}
No competing interests to declare.

\printbibliography

\end{document}